\documentclass[11pt,leqno]{amsart}
\usepackage[utf8]{inputenc}
\usepackage{bbm}
\usepackage{eulervm}

\usepackage[T1]{fontenc}
\usepackage[a4paper,asymmetric]{geometry}
\usepackage{latexsym,amssymb}
\usepackage{hyperref}
\usepackage{amsmath,amsthm}
\usepackage{amsfonts}
\usepackage[american]{babel}
\usepackage{mathrsfs}
\usepackage{graphicx}
\usepackage{xcolor}
\usepackage{tikz}
\usetikzlibrary{patterns}
\newcommand{\R}{\mathbb{R}}
\newcommand{\rn}{\mathbb{R}^N}
\newcommand{\N}{\mathbb{N}}
\newcommand{\AAA}{{\mathscr{A}}}
\newcommand{\Aa}{{\mathscr{A}^a}}
\newcommand{\Ab}{{\mathscr{A}_b}}
\newcommand{\Aab}{{\mathscr{A}^a_b}}

\newcommand{\bd}{\partial}

\newcommand{\C}{\mathscr{C}_{a,b}}
\newcommand{\CC}{\mathscr{K}_{a,b}}

\newcommand{\ee}{\textsf{e}}
\newcommand{\ep}{{\varepsilon}}

\newcommand{\inr}[1]{\mathscr{I}_{#1}}
\newcommand{\inte}{\textrm{int}}

\newcommand{\jla}[1]{{\textrm{J}}_{#1}}

\newcommand{\la}{\lambda}
\newcommand{\laa}{{\mathscr{L}^a}}
\newcommand{\lab}{{\mathscr{L}^a_b}}
\newcommand{\lbb}{{\mathscr{L}_b}}

\newcommand{\oo}{\Omega}
\newcommand{\oet}{{{\oo^{\eta}}}}
\newcommand{\oola}[1]{\oo_{#1}}
\newcommand{\ola}{{\Omega_\la}}
\newcommand{\ot}{{\widetilde{\Omega}}}
\newcommand{\outr}[1]{\mathscr{O}_{#1}}

\newcommand{\xx}{\bar x}

\newcommand{\trexy}{\begin{minipage}{2 cm} $3\times x=1.0135$\\$y=0.1012$\end{minipage}}
\newcommand{\isotri}{\begin{minipage}{1.5 cm} isosceles\\ triangle \end{minipage}}
\newcommand{\equitri}{\begin{minipage}{1.5 cm} equilateral\\ triangle \end{minipage}}
\newcommand{\qrquadr}{\begin{minipage}{1.5 cm} quasi-\\ regular\\ quadrilater \end{minipage}}
\newcommand{\qrpent}{\begin{minipage}{1.5 cm} quasi-\\ regular\\ pentagone \end{minipage}}
\newcommand{\regpent}{\begin{minipage}{1.5cm} regular\\ pentagone \end{minipage}}
\newcommand{\regex}{\begin{minipage}{1.5cm} regular\\ hexagone \end{minipage}}
\newcommand{\regNgon}{\begin{minipage}{1.5cm} regular\\ $N$-gone \end{minipage}}
\newcommand{\Da}{
   \begin{tikzpicture}[x=2mm,y=2mm]
    \draw (0, 3.5);
    \draw[dashed,very thin] (0,0) circle (3);
    \draw[thick] (0,0) circle (1);
    \end{tikzpicture}
}
\newcommand{\isoscUno}{
     \begin{tikzpicture}[x=2mm,y=2mm]
    \draw (0, 3.5);
    \draw[dashed, very thin] (0,0) circle (3);
    \draw[dashed, very thin] (0,0) circle (1);
    \draw[thick] (-2.828,-1)--(2.828,-1)--(0,1.4)--(-2.828,-1);
    \end{tikzpicture}
}
\newcommand{\isoscDue}{
     \begin{tikzpicture}[x=2mm,y=2mm]
    \draw (0, 3.5);
    \draw[dashed, very thin] (0,0) circle (3);
    \draw[dashed, very thin] (0,0) circle (1);
    \draw[thick] (90:3)--(230:3)--(310:3)--(90:3);
    \end{tikzpicture}
}
\newcommand{\equi}{
     \begin{tikzpicture}[x=2mm,y=2mm]
    \draw (0, 3.5);
    \draw[dashed, very thin] (0,0) circle (3);
    \draw[dashed, very thin] (0,0) circle (1);
    \draw[thick] (90:3)--(210:3)--(330:3)--(90:3);
    \end{tikzpicture}
}
\newcommand{\quadriirreg}{
     \begin{tikzpicture}[x=2mm,y=2mm]
    \draw (0, 3.5);
    \draw[dashed, very thin] (0,0) circle (3);
    \draw[dashed, very thin] (0,0) circle (1);
    \draw[thick] (5.801:3)--(121.99:3)--(238.19:3)--(-5.801:3)--(5.801:3);
    \end{tikzpicture}
}
\newcommand{\squareFigure}{
     \begin{tikzpicture}[x=2mm,y=2mm]
    \draw (0, 3.5);
    \draw[dashed, very thin] (0,0) circle (3);
    \draw[dashed, very thin] (0,0) circle (1);
    \draw[thick] (45:3)--(135:3)--(225:3)--(-45:3)--(45:3);
    \end{tikzpicture}
}
\newcommand{\pentairregUno}{
     \begin{tikzpicture}[x=2mm,y=2mm]
    \draw (0, 3.5);
    \draw[dashed, very thin] (0,0) circle (3);
    \draw[dashed, very thin] (0,0) circle (1);
    \draw[thick] (-4:3)--(4:3)--(90.32:3)--(180.08:3)--(269.83:3)--(-4:3); 
    \end{tikzpicture}
}
\newcommand{\pentag}{
\begin{tikzpicture}[x=2mm,y=2mm]
    \draw (0, 3.5);
    \draw[dashed, very thin] (0,0) circle (3);
    \draw[dashed, very thin] (0,0) circle (1);
    \draw[thick] (10:3)--(82:3)--(154:3)--(226:3)--(298:3)--(10:3);
    \end{tikzpicture}
}
\newcommand{\esag}{
\begin{tikzpicture}[x=2mm,y=2mm]
    \draw (0, 3.5);
    \draw[dashed, very thin] (0,0) circle (3);
    \draw[dashed, very thin] (0,0) circle (1);
    \draw[thick] (0:3)--(60:3)--(120:3)--(180:3)--(240:3)--(300:3)--(0:3);
    \end{tikzpicture}
}
\newcommand{\Db}{
   \begin{tikzpicture}[x=2mm,y=2mm]
    \draw (0, 3.5);
    \draw[thick] (0,0) circle (3);
    \draw[dashed, very thin] (0,0) circle (1);
    \end{tikzpicture}
}
\numberwithin{equation}{section}
\newtheorem{theorem}{Theorem}[section]
\newtheorem{proposition}[theorem]{Proposition}
\newtheorem{corollary}[theorem]{Corollary}
\newtheorem{lemma}[theorem]{Lemma}
\newtheorem{remark}[theorem]{Remark}

\newtheorem{definition}[theorem]{Definition}

\begin{document}

\title[Optimal sets for a class of minimization problems with convex constraints]{Optimal sets for a class of minimization problems with convex constraints}

\author[C. Bianchini, A. Henrot]{Chiara Bianchini, Antoine Henrot}

\address{C. Bianchini, A. Henrot: Institut Elie Cartan, Universit\'e Henri
Poincar\'e Nancy, Boulevard des Aiguillettes B.P. 70239, F-54506 Vandoeuvre-les-Nancy Cedex, France}
\email{chiara.bianchini@iecn.u-nancy.fr}

\email{antoine.henrot@iecn.u-nancy.fr}

\date{}


\begin{abstract}
We look for the minimizers of the functional
$\jla{\la}(\oo)=\la|\oo|-P(\oo)$ among planar convex domains constrained to lie into a given ring. 
We prove that, according to the values of the parameter $\la$, the solutions are either a disc or a polygon. 
In this last case, we describe completely the polygonal solutions by reducing the problem to a finite dimensional optimization problem. 
We recover classical inequalities for convex sets involving area, perimeter and inradius or circumradius and find a new one.
\end{abstract}

\maketitle

\section{Introduction}
Shape optimization problems for geometric functionals as the volume and the perimeter have always aroused a large interest; the most famous examples are inequalities of the isoperimetric type. 
In particular in the classical isoperimetric inequality one looks for a set minimizing the perimeter among all the sets of fixed area or, equivalently, for a set maximizing the area among all the sets of fixed perimeter. 
On the other hand one can consider reverse isoperimetric type inequalities. 
Of course, this makes sense only working with supplementary constraints like convexity or involving inradius and/or circumradius in order to avoid degenerate solutions. 
Namely one can maximize the perimeter among convex sets with fixed volume contained in some given ball or, analogously, minimize the
volume among sets of fixed perimeter which contain a given ball. 
The analysis of such classical problems naturally leads to the study of critical points of functionals of the type 
\begin{equation}\label{Jla1}
\jla{\la}(\oo)=\la|\oo|-P(\oo),
\end{equation}
where $|\cdot|$ is the area, $P(\cdot)$ is the perimeter and $\la$ stands for some Lagrange multiplier. 

Another motivation is to get geometric inequalities for convex sets like in \cite{F} or \cite{Kr} (see \cite{SA} for a good overview of such inequalities).
In particular in \cite{F} J. Favard investigated some functionals of the area and the perimeter which are homogeneous in $P$ and $|\cdot|^{1/2}$; in particular he studied the maximum for the functional $P(\oo)/\sqrt{|\oo|}$ among convex sets contained in an annular ring and he proved that the optimal set is a polygon which is inscribed in the exterior ball and all of its sides, except at most one, are tangent to the interior disk. 
The same functional had been investigated by K. Ball in \cite{B} where he presents a reverse isoperimetric inequality in the $N$-dimensional case substituting the constraints on the inradius and circumradius by considering classes of affine equivalent convex bodies, rather than individual bodies.
In particular he proved that for any convex set $K\subseteq \rn$ there exists an affine image $F(K)$ for which 
$$
\frac{P(F(K))}{|F(K)|^{\frac {N-1}N}},
$$
is no larger than the corresponding expression for a regular $N$-dimensional tetrahedron.

In this paper we choose to consider the following minimization problem for every value of the parameter $\la\ge 0$:
\begin{equation}\label{PB}
\min_{\oo\in\C} \la|\oo|-P(\oo),
\end{equation}
where:
$$
\C=\{K\subseteq \R^2 \ \ K \text{ convex, }D_a\subseteq K \subseteq D_b\};
$$
(here and later $D_r$ is the ball of radius $r$ with center at the origin). 
Notice that the class $\C$ is compact with respect to the Hausdorff distance, moreover the functional $ \la|\oo|-P(\oo)$ is bounded from below by $\la|D_a|-P(D_b)$, and continuous thanks to the convexity constraint (see e.g. \cite{HP}); hence the minimum in (\ref{PB}) is in fact achieved for every value of $\la\ge 0$. 
For a more general existence result for minimum problems in the class of convex sets, we refer to \cite{BG}.

In the paper we present a description of optimal sets to Problem (\ref{PB}); more precisely we prove the following result.
\begin{theorem}\label{main}
For every $\la\ge 0$ there exists an optimal set $\ola$ which solves Problem (\ref{PB}).
In particular
\begin{itemize}
\item if $0\le\la\le \frac 1{2b}$ then $\ola=D_b$;
\item if $\frac 1{2b}<\la<\frac 2a$ then $\ola$ is a polygon;
\item if $\la> \frac 2a$ then $\ola=D_a$.
\end{itemize}
\end{theorem}
The proof of this result can be found in Corollary \ref{corollaryPolyg} for the case  $\frac 1{2b}<\la<\frac 2a$, and in Theorem \ref{teoDaDb} for $\la\le\frac 1{2b}$ or $\la\ge \frac 2a$. 
The case of $\la=2/a$ is discussed in details in Remark \ref{la=2/a}.
A further description of the optimal polygon(s) is presented in Section \ref{SecNumbersOfSides}.
Notice that, obviously, the functional is invariant under rotations, thus there is no uniqueness of solution.
Nevertheless we will see that, except for a finite number of values for $\la$, the solution is unique up to rotation.

In order to prove that solutions to Problem (\ref{PB}) are either polygons  or the given balls $D_a$ or $D_b$, the idea is to analyse optimality conditions for (\ref{PB}) either from a geometric or from an analytic point of view. 
In particular the notion of \emph{support function} of the set $K$ will be useful: $h=h_K$ is the function $h:\R^2\to\R$ such that
\begin{eqnarray*}
h_K(u)&=&\sup_{x\in K}<x;u>\quad\text{ for every }u\in\R^2.
\end{eqnarray*}

We consider the functional $\jla{\la}$ defined in (\ref{Jla}), on the class of convex subsets of $\R^2$; hence Problem (\ref{PB}) can be rewritten as 
$$
\min_{\oo\in\C} \jla{\la}(\oo).
$$
Moreover,  the functional $\jla{\la}$ can be rewritten in terms of its support function as follows:
$$
\jla{\la}(\oo)= \frac{\la}2 \int_0^{2\pi} (h^2-h'^2)\,d\theta -\int_0^{2\pi}h\,d\theta.
$$
Recalling that the convexity of a set $K$ can be expressed in terms of its support function as $h''_K+h_K\ge 0$, the class $\C$ is reduced to 
$$
\C=\{ K\subseteq\R^2\ :\ a\le h_K\le b,\ h_K''+h_K \ge 0 \text{ for every }\theta\in[0,2\pi]\}.
$$

A fundamental preliminary result is expressed in theorem below, which is due to J. Lamboley and A. Novruzi (see \cite[Theorem {2.1}]{LN}). 
They considered generic functionals of the form 
$$
\int_0^{2\pi}G(\theta,u(\theta),u'(\theta))d\theta, 
$$
where $u$ stands either for the support function or the gauge function of a planar convex domain, and they proved that, under a concavity property of $G(\theta,u,p)$ solutions to the associated minimum problem are (locally) polygons. 
Applying their result to the formulation of $\jla{\la}$ in terms of support function, we get the following.
\begin{theorem}[\cite{LN}]\label{teoJimmy}
For every $\la\ge 0$, if $\ola$ is a solution to (\ref{PB}) then $\ola$ is locally a polygon in the interior of the annulus $D_b\setminus D_a$.
\end{theorem}
Moreover, using \cite[Theorem {2.2}]{LN}, it is possible to get a range of values of $\la$ for which solutions are polygons.
However, the application of their result yields a range of value $\frac 1b\le\la\le\frac 1a$ while we are able to get the same result for $\frac 1{2b}<\la<\frac 2a$. 
The reason is the following: we actually consider more general perturbations of a convex set that they did.
Namely in the proof of Theorem \ref{teopolyg} we consider perturbations of a generic set $\oo$ of the form $\oet$, expressed by the support functions as
$$
h_{\oet}(\theta)=h_\oo(\theta)+w(\theta,\eta),
$$
with
$$
w(\theta,\eta)=\Big( h_{T_\eta}(\theta)-h_\oo(\theta) \Big)
\chi_{(0,\eta)}(\theta)\quad\text{ or }\quad w(\theta,\eta)=\Big( h_{S_\eta}(\theta)-h_\oo(\theta) \Big)\chi_{(0,2\eta)}(\theta),
$$
where $T_\eta$ is the triangle of vertices $(0,0),(b,0),(b\cos\eta,b\sin\eta)$ and $S_\eta$ is the quadrilateral of vertices
$(0,0),(a,0),(a,a\tan\eta),(a\cos2\eta,a\sin2\eta)$ (see Figure \ref{step1} for details). 
These kind of perturbations are not of the simple type $h_{\oet}(\theta)=h_\oo(\theta)+t \eta(\theta)$ considered in \cite{LN}.


In Section \ref{SecNumbersOfSides} a detailed characterization of optimal polygons is presented.
In particular it is shown that optimal polygons are either inscribed in the exterior ball $D_b$ or circumscribed to the interior ball $D_a$.
This is proved via refinements of a natural geometric argument of ``anti-symmetrization''.
It is in fact evident that an optimal polygon $\oo$ cannot contain two consecutive \emph{free} sides, that is two consecutive sides which are neither a chord of $D_b$ nor tangent to $D_a$. 
Otherwise the perturbation in Figure \ref{pcmov} would be possible, in contradiction with the optimality of the set $\oo$. 
More precisely, assume there exist two free sides $\overline{AB}, \overline{BC}$; we consider the set $\oo_t$ obtained as a perturbation of the set $\oo$ by moving the vertex $B$ in the direction $v=\overrightarrow{AC}$ for a time $t\in\R$ (notice that all the other vertices are fixed). 
\begin{figure}[h]
\centering
\begin{tikzpicture}[x=0.4mm,y=0.4mm]
\draw[very thin, dashed] (-53,0)--(53,0); 
\draw[very thin] (20,0) arc(0:180:20) ;
\draw[very thin] (50,0) arc(0:180:50);
\draw[thick] (-30,0) node(A)[below]{$A$}--(-10,35)node(B)[above]{$B$}--(40,0) node(C)[below]{$C$};
\draw (-23,17) node[left]{$\oo$};
\draw[very thick, ->] (0,35)--(10,35); \draw (5,35) node[above]{$v$};
\draw (0,-10) node{$t=0$};
\begin{scope}[xshift=-150]
\draw[very thin, dashed] (-53,0)--(53,0); 
\draw[very thin] (20,0) arc(0:180:20) ;
\draw[very thin] (50,0) arc(0:180:50);
\draw[thick] (-30,0) node(A)[below]{$A$}--(-17,35)node(B)[above right]{$B_t$}--(40,0) node(C)[below]{$C$};
\draw (-25,17) node[left]{$\oo_t$};
\draw (0,-10) node{$t<0$};
\end{scope}
\begin{scope}[xshift=155]
\draw[very thin, dashed] (-53,0)--(53,0); 
\draw[very thin] (20,0) arc(0:180:20) ;
\draw[very thin] (50,0) arc(0:180:50);
\draw[thick] (-30,0) node(A)[below]{$A$}--(0,35)node(B)[above]{$B_t$}--(40,0) node(C)[below]{$C$};
\draw (-15,17) node[left]{$\oo_t$};
\draw (0,-10) node{$t>0$};
\end{scope}
\end{tikzpicture}
\caption{A parallel chord movement: optimal sets cannot have ``free'' sides.}\label{pcmov}
\end{figure}
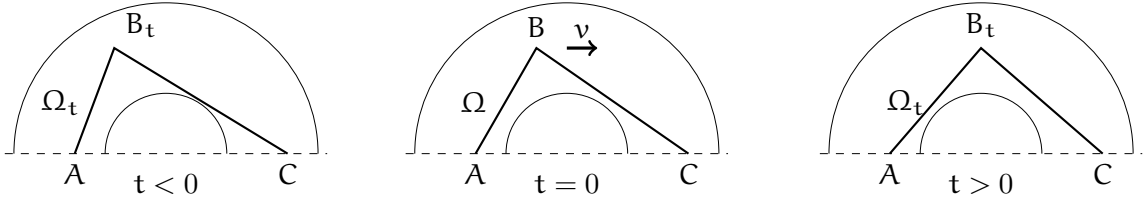
This is a so called \emph{parallel chord movement}, as $\oo_t$ is obtained from $\oo$ by moving its lines (only those contained into the half plane determined by the line $AC$ and the point $B$), along the direction $v$.
For small times the set $\oo_t$ is still a convex set and in particular it still belongs to the class $\C$.
Moreover it is clear that $|\oo_t|=|\oo|$ for every $t\in\R$ and that there exists $\bar t$ such that $P(\oo_{\bar t})>P(\oo)$; hence $\oo$ cannot be optimal.


\section{{Main results}}
\subsection{First characterizations}
\begin{theorem}\label{teopolyg}
Let $\ola$ be a minimizer of (\ref{PB}), then for $1/2b<\la<2/a$, $\bd\ola$ does not contain neither arcs of $D_a$ nor arcs of $D_b$.
\end{theorem}

\begin{corollary}\label{corollaryPolyg}
For every $1/2b<\la<2/a$ minimizers to (\ref{PB}) are polygons.
\end{corollary}
\begin{proof}
By Theorem \ref{teoJimmy} for every value of $\la\ge 0$ a minimizer
can be composed only by segments and arcs of $D_a$ and $D_b$. We
will prove in Corollary \ref{Nfinite} that the number of segments is
necessarily finite. Thus using Theorem \ref{teopolyg} the thesis
follows.
\end{proof}

\begin{proof}[Proof of Theorem \ref{teopolyg}]
We split the proof into two steps.
\medskip

Step 1: if $\la>1/2b$ then $\bd\ola$ does not contain arcs of $\bd D_b$.
\smallskip

Let $\oo\in\C$ and assume that it contains an arc of $\bd D_b$ on its  boundary, that is there exists a subinterval of  $[0,2\pi)$ (which for simplicity is assumed to be $(0,\gamma)$ for some $\gamma>0$), such that
$$
\{\theta\in[0,2\pi)\ : \ h_{\oo}(\ee^{i\theta})=b\} \supseteq (0,\gamma).
$$
Let $\eta\in(0,\gamma/2)$ be such that $\cos\eta\ge a/b$ and consider $\oet$ obtained from $\oo$ by cutting a part of the arc by a chord of central angle $\eta$ (see Figure \ref{step1} (a)).
%
%
Notice that, as we choose $\cos\eta\ge a/b$, the new set $\oet$ still belongs to the class $\C$.
\begin{center}
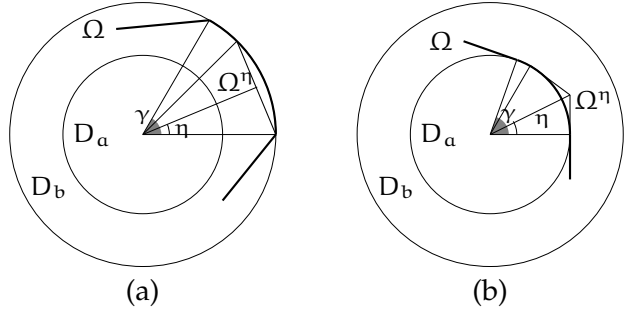
\begin{figure}[h]
\begin{tikzpicture}[x=0.35mm,y=0.35mm]
\draw[very thin] (0,0) circle (50);
\draw (-46,-20) node [right] {\small$D_b$};
\draw (-30,0) node [right] {\small$D_a$};
\draw[very thin] (0,0) circle (30);
\fill[gray] (0,0) -- (7,0) arc (0:60:7)-- cycle;
\draw[very thin] (0,0)--(50,0);
\draw[very thin] (0,0)--(25,43.301270189); 
\draw (0,0) node [above]{\scriptsize$\gamma$};
\draw[thin] (50,0)--(35.355339059,35.355339059);
\draw[very thin] (0,0)--(35.355339059,35.355339059);\draw[very thin] (0,0)--(43,18); 
\draw (10,0) arc (0:22.5:10);
\draw (5:15) node{\scriptsize$\eta$};
\draw[thick] (-10,40) node [left] {\small$\oo$}--(25,43.301270189); 
\draw[thick] (50,0)--(30,-25);
\draw[thick] (50,0) arc (0:60:50);
\draw (35,20) node {\small$\oet$};
\draw (0,-60)node{(a)};
\begin{scope}[xshift=130]
\draw[very thin] (0,0) circle (50);
\draw (-46,-20) node[right] {\small$D_b$};
\draw (-30,0) node [right] {\small$D_a$};
\draw[very thin] (0,0) circle (30);
\fill[gray] (0,0) -- (7,0) arc (0:70:7)-- cycle;
\draw[very thin] (0,0)--(30,0);
\draw[very thin] (0,0)--(10, 28.284271247); 
\draw (7,7) node {\scriptsize$\gamma$};
\draw[thin] (30,0)--(30,15);
\draw[very thin] (0,0)--(30,15); 
\draw[very thin] (10,0) arc (0:30:10);
\draw (13,5) node [right]{\scriptsize$\eta$};
\draw[very thin] (0,0) -- (15, 25.980762114); 
\draw[thin] (15,26.2) -- (30,15);
\draw[thick] (-10,35.355339059) node [left] {\small$\oo$}--(10,28.284);
\draw[thick] (30,-17)--(30,0);
\draw[thick] (30,0) arc (0:70:30);
\draw (28,13) node[right]{\small$\oet$};
\draw (0,-60) node{(b)};
\end{scope}
\end{tikzpicture}
\caption{The constructions in Step 1 and Step 2 respectively.}\label{step1}\label{step2}
\end{figure}
\end{center}
We want to show that $\jla{\la}(\oo)>\jla{\la}(\oet)$; we get
\begin{equation}\label{arch-chordb}
\jla{\la}(\oo)-\jla{\la}(\oet)= \frac b2 (\eta-\sin\eta\cos\eta)\left(\la 2 b- 4\frac{\eta-\sin\eta}{\eta-\sin\eta\cos\eta}\right),
\end{equation}
for every $\eta\in(0,\gamma/2)$ sufficiently small.
As $lim_{\eta\to 0}4\frac{\eta-\sin\eta}{\eta-\sin\eta\cos\eta}=1$ and $\la>1/2b$, for $\eta$ sufficiently small we get
$$
\la 2 b- 4\frac{\eta-\sin\eta}{\eta-\sin\eta\cos\eta}>0,
$$
which gives the desired result.

\medskip

Step 2: if $\la<2/a$ then $\bd\ola$ does not contain arcs of $\bd D_a$.
\smallskip

Consider $\oo\in\C$ and assume that $\bd\oo$ contains an arc of $\bd
D_a$, that is  there exists an subinterval of  $[0,2\pi)$ (which for
simplicity is assumed to be $(0,\gamma)$ for some $\gamma>0$), such
that
$$
\{\theta\in[0,2\pi)\ : \ h_{\oo}(\ee^{i\theta})=a\} \supseteq (0,\gamma).
$$
Let $\eta\in(0,\gamma/2)$ be such that $\cos\eta\ge a/b$ and
consider $\oet$  obtained from $\oo$ by cutting a part of the arc of
$D_a$ of width equals to $2\eta$ by two tangent lines to $D_a$, as
shown in Figure \ref{step2} (b).
%
%
Notice that, choosing $\eta>0$ such that $\cos\eta\ge a/b$, the set $\oet$ still belongs to the class $\C$.
Moreover, comparing $\jla{\la}(\oet)$ and $\jla{\la}(\oo)$ we obtain
$$
\jla{\la}(\oo)-\jla{\la}(\oet)=-a^2(\tan\eta-\eta)(\la-\frac 2a),
$$
which is positive as $\la<2/a$ and hence $\bd\ola$ cannot contain arcs of $D_a$ for every $\la<2/a$.
\end{proof}

\subsection{Reduction to an optimization problem of finite dimension}
We define three classes of segments which will be useful in what
follows. In particular it will turn out that the sides of an optimal
polygon necessarily belong to these classes; as already noticed, in fact, free sides are not allowed for an optimal polygon.
We here prove that in fact they are necessarily either chord of $D_b$ or tangent side to $D_a$.

A similar representation for convex sets in terms of their central angles has been used also for other type of functionals in \cite{C}.
\begin{definition}

The class $\laa$ represents the class of tangent sides to $D_a$
which are not chords of $D_b$. In particular if $P_iP_j$ and
$P_jP_k$ are segments tangent to $D_a$,  with $P_i,P_k\in\bd D_a$,
the segments $P_iP_j$ and $P_jP_k$ are identified in the class
$\laa$ as the same element (and hence they are counted only once).

The class $\lbb$ represents the class of segments which are chords of $D_b$ not tangent to $D_a$.
In particular the elements of $\lbb$ are half chords and each couple of half chords is in fact identified in the same element of $\lbb$.
Hence for each chord $P_iP_j$ of $D_b$ if $Q_i$ is its medium point, the segments $P_iQ_i$ and $Q_iP_j$ are identified in class $\lbb$.

The class $\lab$ represents the class of segments which are at the same time tangent to $D_a$ and chords of $D_b$.
In particular a segment $P_iP_j$ belongs to $\lab$ if $P_i\in\bd D_b$ and $P_j\in\bd D_a$.
Again we will count these segments in couples (it will be clear later that in fact the number of these segments is always even).
\end{definition}
In an analogous way we define the corresponding classes of central angles.
\begin{definition}

The class $\Aa$ is the class of angles which determine a segment in $\laa$.

The class $\Ab$ is the class of angles which determine a segment in $\lbb$.

The class $\Aab$ is the class of angles which determine a segment in $\lab$.
\end{definition}
%
%
%
\begin{figure}[h]
\centering
\begin{tikzpicture}[x=0.5mm,y=0.5mm]
\draw (0,0) circle (40);
\draw (160:45) node [above] {\small$D_b$};
\draw (230:10) node [left] {\small$D_a$};
\draw (0,0) circle (20);
\draw[thick] (210:40)--(-30,5.77350269)--(3.29048,24.99377)--(20,8.284)--(20,-25)--(-95:40) arc (-95:-150:40);
\draw[very thin] (0,0)--(-10,17.3205) node[above]{\small$P_5$} ;
\draw[very thin] (0,0)--(3.29048,24.99377) node[above]{{\small$P_4$}};
\draw[very thin] (0,0)--(14.14,14.14)node[above right]{{\small{$P_3$}}};
\draw[fill=gray!40] (0,0)-- +(45:8) arc (45:82:8) -- cycle;
\draw (5,5) node [above] {$\scriptstyle\theta_1$};
\draw[very thin] (0,0)--(20,8.284) node[right]{\small$P_2$};
\draw[very thin] (0,0)--(20,0) node[right]{\small$P_1$};
\draw[fill=gray] (0,0) -- (8,0) arc (0:22:8)-- cycle;
\draw (13,2)  node {$\scriptstyle\theta_2$};
\draw (0,-50) node{(b)};
\begin{scope}[xshift=-150]
\draw (0,0) circle (40);
\draw (-25,-30) node [left] {\small$D_b$};
\draw (0,-10) node [left] {\small$D_a$};
\draw (0,0) circle (20);
\draw[thick] (-35,-10)--(-40,0) node[left]{\small$P_4$}--(3.29048,24.99377)--(20,8.284)--(20,-34.64101615)node[right]{\small$P_1$}-- (-5,-30)--(-35,-10);
\draw[very thin] (0,0)--(-40,0);
\draw[very thin] (0,0)--(-10,17.3205)node[above]{\small$P_3$};
\draw[fill=gray!15] (0,0) -- (180:8) arc (180:120:8)-- cycle;
\draw (-12,0) node [above] {$\scriptstyle\xi_0$};
\draw[very thin] (0,0)--(20,0) node[right]{\small$P_2$};
\draw[very thin] (0,0)--(20,-34.64101615);
\draw (0,-50) node{(a)};
\end{scope}
\begin{scope}[xshift=155, rotate=100]
\draw (0,0) circle (40);
\draw (100:35) node {\small$D_b$};
\draw (0,0) circle (20);
\draw (10,10) node {\small$D_a$};
\draw[thick] (25,-17)--(20,-34.64101615)--(-10.352761804,-38.637033052)--(-40,0)--(114:20) arc (114:15:20) -- (25,-17); 
\draw[thin] (0,0)--(-40,0) node[below right]{\small$P_1$} ;
\draw[thin] (0,0)--(-10.352761804,-38.637033052) node[right]{\small$P_2$};
\draw[thin] (0,0)--(20,-34.64101615) node[right]{\small$P_3$};
\draw[very thin] (0,0)--(-25.175,-19.32) node[below]{\small$Q_1$};
\draw[fill=gray] (0,0) -- (180:8) arc (180:218:8)-- cycle;
\draw (-13,-4)  node {$\scriptstyle\eta_1$};
\draw[very thin] (0,0)--(4.82365,-36.639) node[right]{\small$Q_2$};
\draw[fill=gray!50] (0,0)--(278:8) arc (270:292:8) -- cycle;
\draw (4,-8)  node[right] {$\scriptstyle\eta_2$};
\end{scope}
\draw (110,-50) node{(c)};
\end{tikzpicture}
\caption{The classes of segments $\lab$, $\laa$, $\lbb$ and the corresponding classes of angles $\Aab$, $\Aa$, $\Ab$.}\label{classesL}
\end{figure}
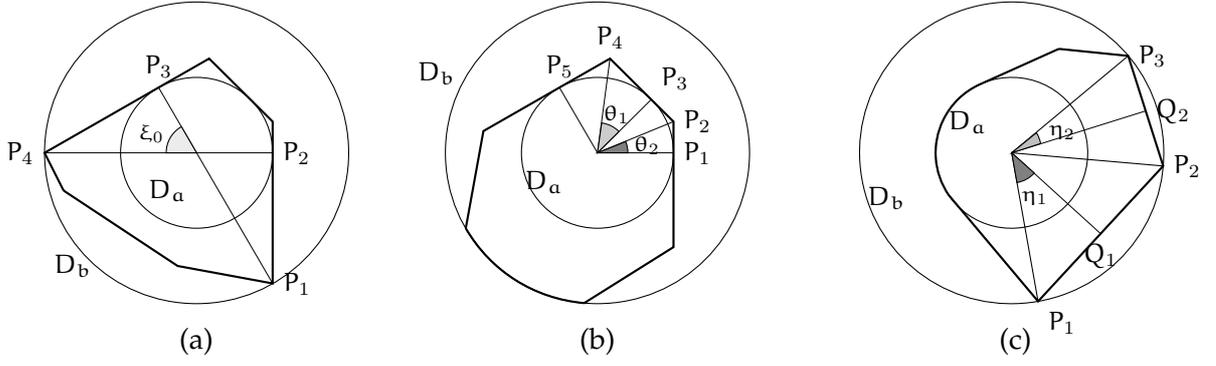
\begin{remark}
Figure \ref{classesL}, (a), represents elements $\xi_0$ in  $\Aab$
and the corresponding segments $P_1P_2\equiv P_3P_4$ in $\lab$; in
particular each couple of segments and angles are identified, so
that in the example it holds $|\Aab|=|\lab|=1$.

Figure \ref{classesL}, (b), represents elements $\theta_i$ in  the
class $\Aa$ and the corresponding segments $P_1P_2\equiv P_2P_3,
P_3P_4\equiv P_4P_5$ in the class $\laa$; in the example it holds
$|\Aa|=|\laa|=2$.

Figure \ref{classesL}, (c), represents elements $\eta_j$ in the
class $\Ab$ and the corresponding segments $P_kQ_k$ in the class
$\lbb$; as each couple of segments $P_iQ_i,Q_iP_{i+1}$ is
identified, in the example it holds $|\Ab|=|\lbb|=2$.
\end{remark}
Notice that all the segments in the class $\lab$ have the same
length equal to $\sqrt{b^2-a^2}$ and analogously each angle
$\xi_0\in\Aab$ has the same value:
\begin{equation}\label{xi0}
\sin\xi_0=\frac{\sqrt{b^2-a^2}}b, \qquad \cos\xi_0=\frac ab.
\end{equation}
Moreover for every $L_i\in\laa$ there exists $\theta_i\in\Aa$ such that $L_i=a\tan\theta_i$ with $\theta_i<\xi_0$, while for $L_j\in\lbb$ there exists $\eta_j\in\Ab$ such that $L_j=b\sin\eta_j$ and $\eta_j<\xi_0$.

By construction it always holds
$$
0<\theta_i,\eta_j < \xi_0 <\frac {\pi}2,
$$
moreover by convexity $\sum_{x\in\Aa\cup\Ab\cup\Aab}x\le\pi$ and $\sum_{l\in \lab\cup\laa\cup\lbb}l\le P(\oo)/2$.
More precisely for an optimal polygon $\oo$, equality holds in the previous expressions, as shown in the following crucial theorem.
\begin{theorem}\label{teolaalablb}
Let $\ola$ be a solution to (\ref{PB}) then its boundary can be decomposed into unions of arches of $\bd D_a$ and $\bd D_b$ and segments $L_i$ belonging to
$\lab\cup\laa\cup\lbb$.
\end{theorem}
Thanks to this result an optimal polygon $\oo$ can be characterized by its classes of segments $\lab,\laa,\lbb$ or, analogously, by its classes of central angles $\Aab,\Aa,\Ab$.
In particular by construction it turns out that if $\bd\oo$ is composed only by arcs of $D_a$ and $D_b$ and segments in the classes $\lab,\laa,\lbb$, then the number of segments which have one vertex on $\bd D_b$ and the other one on $\bd D_a$ (that is the segments which identify the class $\lab$), is even and hence we are allowed to identify segments of the type $\sqrt{b^2-a^2}$ in couple.

\begin{definition}

We define the class $\CC$ as the class of sets $\oo$ such that $D_a\subseteq \oo\subseteq D_b$ and $\bd\oo=\cup_{i\in I}L_i$, with $L_i\in\lab\cup\laa\cup\lbb$.
\end{definition}
Hence, for every $\oo\in\CC$, the functional $\jla{\la}(\oo)$ can be expressed as:
\begin{eqnarray}\label{Jla}
\jla{\la}(\oo)&=&\la\Big( \sum_{\xi_0\in\Aab} a^2\tan\xi_0 +a^2\,\sum_{\theta_i\in\Aa}\tan\theta_i+b^2\,\sum_{\eta_j\in\Ab}\sin\eta_j\cos\eta_j \Big) \\
        &&- 2\Big( \sum_{\xi_0\in\Aab}a\tan\xi_0 +a\, \sum_{\theta_i\in\Aa}\tan\theta_i+b\,\sum_{\eta_j\in\Ab}\sin\eta_j \Big). \nonumber
\end{eqnarray}
Notice that $\CC\subseteq\C$, that is each $\oo$ in the class $\CC$ is a convex polygon.
Hence by Corollary \ref{corollaryPolyg} and Theorem \ref{teolaalablb} it follows
$$
\min_{\oo\in\C}\jla{\la}(\oo)=\min_{\oo\in\CC}\jla{\la}(\oo),
$$
for every $1/2b<\la< 2/a$.
In particular for such values of $\la$ the minimum problem can be expressed as:
\begin{eqnarray}\label{JlaCC}
&& \min_{\oo\in\CC}\jla{\la}(\oo) = \nonumber\\
&&\qquad \min\left\{\jla{\la}(\oo)\ |\ \oo\in\CC;\ \sum_{\xi_0\in\Aab}\xi_0+\sum_{\theta_i\in\Aa}\theta_i+\sum_{\eta_j\in\Ab}\eta_j=\pi;\quad 0<\theta_i, \eta_j<\xi_0 \right\}.
\end{eqnarray}

Notice that the classes $\Aab,\Aa,\Ab$ do not identify a unique shape of polygon, as shown in Figure \ref{memeangles}.
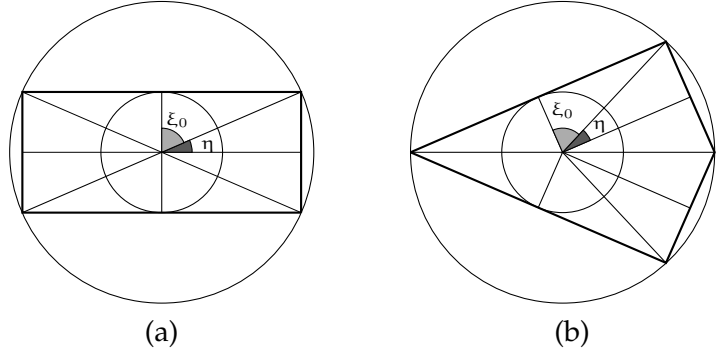
\begin{figure}[h]
\centering
\begin{tikzpicture}[x=0.4mm,y=0.4mm]
\draw (0,0) circle (50);
\draw (0,0) circle (20);
\draw[thick] (-23.58:50)--(23.58:50)--(156.42:50)--(203.58:50)--(-23.58:50);
\draw[very thin] (-23.58:50)--(156.42:50); \draw[thin] (23.58:50)--(203.58:50);
\draw[very thin] (-45.82,0)--(45.82,0);\draw[thin] (0,-20)--(0,20);
\draw[fill=gray!130] (0,0) -- (0:10) arc (0:23.58:10)-- cycle;
\draw (11:10)  node[right] {\tiny$\eta$};
\draw[fill=gray!65] (0,0) -- (90:8) arc (90:23.58:8)-- cycle;
\draw (65:12) node{\tiny$\xi_0$};
\draw (0,-60) node{(a)};
\begin{scope}[xshift=150, rotate=90]
\draw (0,0) circle (50);
\draw (0,0) circle (20);
\draw[thick] (90:50)--(222.84:50)--(270:50)--(317.16:50)--(90:50);
\draw[thin] (90:50)--(270:50); \draw[thin] (0,0)--(222.84:50);\draw[thin] (0,0)--(317.16:50);
\draw[thin] (0,0)--(246.23:46);\draw[thin] (0,0)--(293.58:46);
\draw[thin] (0,0)--(156.42:20);\draw[thin] (0,0)--(23.58:20);
\draw[fill=gray!130] (0,0) -- (317.16:10) arc (317.16:293.58:10)-- cycle;
\draw (305:15)  node {\tiny$\eta$};
\draw[fill=gray!65] (0,0) -- (-42.84:8) arc (-42.84:23.58:8)-- cycle;
\draw (0:14)  node {\tiny$\xi_0$};
\end{scope}
\draw (135,-60) node{(b)};
\end{tikzpicture}
\caption{Two different polygons corresponding to the same classes of central angles.
For them the value of the functional $\jla{\la}$ is the same}\label{memeangles}
\end{figure}
However the value of $\jla{\la}$ only depends on the values of the angles and their belonging to a certain class; indeed these possible different polygons are equivalent for the minimization problem. 
Hence in what follows we will refer to a certain polygon $\oo$ regarding only its classes of central angles (or equivalently its classes of segments).

\begin{proof}[Proof of Theorem \ref{teolaalablb}]
Thanks to Theorem \ref{teoJimmy} it is enough to prove that each segment of $\bd\ola$ belongs to $\lab\cup\laa\cup\lbb$.
Assume there exists a side $PQ$ which is neither tangent to $D_a$ nor a chord of $D_b$ with $Q\in\inte D_b\setminus \overline{D_a}$.
We define the point $H\in\bd \ola$ such that $HQ\in\bd\ola$ and $OH\perp HQ$, as shown in Figure \ref{oep}.
Let $\eta$ be the angle determined by the normal lines to $HQ$ and $QP$, respectively.
\begin{center}
\begin{figure}[h]
\begin{tikzpicture}[x=0.4mm,y=0.4mm]
\draw[very thin] (0,0) circle (67);\draw (-60,-30) node [right] {\small$D_b$};
\draw[very thin] (0,0) circle (20);\draw (-20,0) node [left] {\small$D_a$};
\draw (0,0) node[left]{$O$};
\draw (0,40) node [above] {\small$H$};
\draw[thick] (-30,40) -- (40,40);   
\draw (40,40) node [above] {\small$Q$};
\draw[thick] (40,40) -- (20,-30) node[right] {\small$P$}; 
\draw[red] (20,-30)--(250:20);\draw[thick] (20,-30)--(-20,-13);
\draw (-30, 40) node [above]{\small$\ola$};
\draw (-30,40) -- (50,40);
\draw (30,40)  -- (20,-30); 
\draw (50,40)  -- (20,-30); 
\draw[<->,very thin] (30,38) -- (39, 38);  
\draw (33,40) node [below] {$\scriptstyle\ep'$};
\draw[<->,very thin] (40,38) -- (49, 38);
\draw (43,40) node [below] {$\scriptstyle\ep''$};
\draw (20,27) node {\small$\oola{\la}^{\scriptscriptstyle\ep'}$};
\draw (52,18) node {\small$\oola{\la}^{\scriptscriptstyle\ep''}$};
\draw[very thin] (0,0) -- (0,40);
\draw[very thin] (0,0) -- (26.41509434,-7.547169811);
\draw (12,12) node [below] {$\eta$};
\draw[fill=gray!60, very thin] (0,0) -- (-18:9) arc (-18:90:9)-- cycle;
\draw[very thin] (0,0)--(250:20);
\end{tikzpicture}
\caption{Segments of optimal polygons necessarily belong to $\lab\cup\laa\cup\lbb$.}\label{oep}
\end{figure}
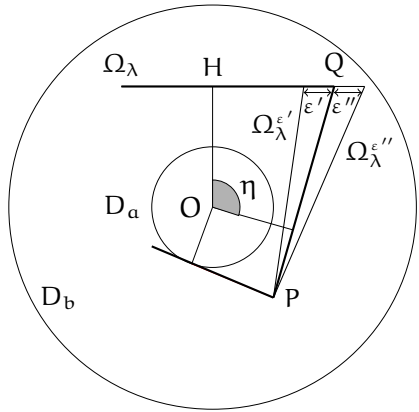
\end{center}
We consider $\oola{\la}^\ep$ a perturbation of $\ola$ obtained slightly  moving the vertex $Q$ in a position $Q^\ep$, which belongs to the same line $HQ$ and which is at distance $\ep$ from $Q$ (see Figure \ref{oep}). 

In the case of a perturbation with positive $\ep$, we have
$$
\jla{\la}(\oola{\la}^\ep)-\jla{\la}(\ola)={\ep}\, \sin\eta\;\Big(\frac{\la}2\;\overline{QP}-\frac{1-\cos\eta}{\sin\eta}+\frac{o(\ep)}{\ep}\Big),
$$
which implies, by the optimality of $\ola$,
\begin{equation}\label{ep>0}
\frac{\la}2\ge\frac{\tan{\eta/2}}{\overline{QP}}.
\end{equation}
In an analogous way, for $\ep<0$ we get
$$
\jla{\la}(\oola{\la}^\ep)-\jla{\la}(\ola)={-\ep}\,\sin\eta\;(\frac{\la}2\;\overline{QP}-\frac{1-\cos\eta}{\sin\eta}+\frac{o(\ep)}{\ep}),
$$
which entails
$$
\frac{\la}2\le\frac{\tan{\eta/2}}{\overline{QP}},
$$
and hence, by condition (\ref{ep>0}) we get, as a necessary condition for the optimality of $\ola$,
$$
\la= 2\, \frac{\tan{\eta/2}}{\overline{QP}}.
$$
Let us now show that, even in this case, such a set $\ola$ cannot be a minimizer.

Fix $\bar\la= 2\; \tan\textstyle{\frac{\eta}2}\,/\;\overline{QP}$. We consider the
same perturbation as before, for $\ep>0$ and again we assume $\ep$
small enough in such a way that $\oola{\bar\la}^\ep$ still belongs
to $\C$. We compute
$\jla{\bar\la}({\oola{\bar\la}^\ep})-\jla{\bar\la}(\oola{\bar\la})$
in order to show that
$\jla{\bar\la}({\oola{\bar\la}^\ep})<\jla{\bar\la}(\oola{\bar\la})$,
and hence that $\oola{\bar\la}$ cannot be a minimizer.
\begin{eqnarray}
\jla{\bar\la}({\oola{\bar\la}^\ep})-\jla{\bar\la}(\oola{\bar\la}) &=&
\sin\eta\;\frac{1-\cos\eta}{\sin\eta\;\overline{QP}}\, \overline{QQ^\ep}\,\overline{QP}-\overline{QQ^\ep}-\overline{Q^\ep P}+\overline{QP}\nonumber\\
        &=& \overline{QP}-\cos\eta\overline{QQ^\ep}-\sqrt{\overline{QP}^2+\overline{QQ^\ep}^2-2\cos\eta\,\overline{QP}\,\overline{QQ^\ep}},\label{QM}
\end{eqnarray}
notice that the quantity (\ref{QM}) is always negative for every
positive $\ep$  as, if $\overline{QP}-\cos\eta\,\overline{QQ^\ep}$
is non negative, it holds
\begin{eqnarray*}
\overline{QP}-\cos\eta\,\overline{QQ^\ep} &=& \sqrt{\overline{QP}^2-2\cos\eta\,\overline{QP}\,\overline{QQ^\ep}+\cos^2\eta\,\overline{QQ^\ep}}\\
        &<& \sqrt{\overline{QP}^2+\overline{QQ^\ep}^2-2\cos\eta\,\overline{QP}\,\overline{QQ^\ep}}.
\end{eqnarray*}
\end{proof}
\begin{remark}
Notice that, as highlighted in the introduction about the proof of Theorem \ref{teopolyg}, the perturbations considered in the
above proof are not of the linear form
$$
h_{\oo_t}(\theta)=h_\oo(\theta)+tv(\theta).
$$
This allows us to get more information about the optimal domains.
\end{remark}
As already noticed, the class $\Aab$ is composed by copies of the same angle $\xi_0$ which depends only on the data $a,b$: $\cos\xi_0=a/b$. 
Hence $\Aab$ has at most $\pi/\xi_0$ elements which in particular implies that it is finite. 
Regarding $\Aa$ and $\Ab$ the following theorem holds which implies in  particular that $\Aa$ and $\Ab$ are also finite sets (see Corollary \ref{Nfinite}).
\begin{theorem}\label{lemmaAaAb}
Let $\ola$ be an optimal set belonging to the class $\CC$ then
\begin{enumerate}
\item\label{thetatheta} for $\la\neq 2/a$ there exists $\theta\in (0,{\textstyle\frac{\pi}2})$ such that if $\Aa$ is not empty, then $\Aa=\{\theta\}$;
\item\label{xxy} there exist $x,y\in (0,{\textstyle\frac{\pi}2})$ such that if $\Ab$ is not empty, then either it is a singleton or $\Ab=\{x,...,x\}$ or $\Ab=\{x,...,x,y\}$ with $x>y$ and $\cos x+\cos y= 1/b\la$.
\end{enumerate}
\end{theorem}
\begin{remark}\label{la=2/a}
In the case $\la=2/a$ the boundary of an optimal set $\ola$  only contains arcs of $D_a$ or segments tangent to $D_a$ as it follows by Theorem \ref{teoJimmy}, Step 1 in Theorem \ref{teopolyg}, Theorem \ref{teolaalablb} and Lemma \ref{Lbvuoto}. 
Hence, for $\la=2/a$, either $\ola=D_a$ or $\ola$ is a circumscribed figure to $D_a$ which possibly has both tangent segments and arcs. 
Indeed for a polygons $\oo$ circumscribed to $D_a$ we have $|\oo|=P(\oo) a/2$ hence $\jla{\scriptstyle{\frac2a}}(\oo)=0$; more generally the same arrives if $\oo$ is circumscribed to $D_a$ and it contains arcs of $D_a$. Hence either $\Aa$ is empty or
$\Aa=\{\theta_1,...,\theta_m\}$ for some $m$ such that
$\sum_{i=1}^m\theta_i\le\pi$ and $\cos\theta_i>a/b$.
\end{remark}
\begin{proof}[Proof of Theorem \ref{lemmaAaAb}]
We analyze first and second order optimality conditions for Problem (\ref{JlaCC}). 
By the formulation (\ref{Jla}) the functional $\jla{\la}$ can in  fact be considered as a function of the angles $\xi_0\in\Aab, \theta_i\in\Aa, \eta_j\in\Ab$.
As their sum is finite and each $\theta_i,\eta_j$ is positive, the sets $\Aa$ and $\Ab$ have at most countably many elements, while $\Aab$ is finite.

Consider $\ola$  and assume $\Aab=\{\xi_0,...,\xi_0\}$ with $|\Aab|=p$,  $\Aa=\{\theta_1,...,\theta_i,...\}$ with $|\Aa|=q_a$, $\Ab=\{\eta_1,...,\eta_j,...\}$ with $|\Ab|=q_b$; let $N=p+q_a+q_b$, possibly infinity. 
Let us indicate by $X\in\R^N$ the sequence of angles
$$
{X}=(\xi_0,...,\xi_0,\theta_1,...,\theta_i,... ,\eta_1,...,\eta_j,...)=(x_k)_{k=1,...,N},
$$
and let $\bar{X}$ be the vector corresponding to the optimal set $\ola$. 
With abuse of notation we write $\jla{\la}(X)$ meaning $\jla{\la}(\oo)$,  where $\oo$ is the set corresponding to $X$. 
As $\oo\in\CC$, $\jla{\la}(\oo)$ can be expressed in the form (\ref{Jla}), under the constraints in (\ref{JlaCC}), namely
$$
g_k(X)=x_k-\xi_0<0\qquad\text{ and }\qquad h(X)=\sum_{i=1}^N x_i -\pi=0.
$$
By the first order optimality conditions there exist Lagrange multipliers $\mu_0\in\R$, $\mu_k\in\R^+$ for $k=1,...,p$ such that
\begin{equation}\label{1optimality}
\begin{cases}
D\jla{\la}(\bar{X}) = \mu_0 Dh(\bar{X})+\sum_{k=1}^{p}\mu_kDg_k(\bar{X}),\\
\sum_{k=1}^{P} \mu_kg_k(\bar{X})=0;
\end{cases}
\end{equation}
this is equivalent to
\begin{equation}\label{1ordine}
\begin{cases}
b^2(\la -\frac{2}a)=\mu_0+\mu_k                   &\qquad\text{for } k=1,...,p \\
a^2(\la-\frac 2a)\frac 1{\cos^2\theta_i} = \mu_0  &\qquad\text{for every }\theta_i\in\Aa \\
\la b^2\cos2\eta_j -2b\cos\eta_j=\mu_0            &\qquad\text{for every }\eta_j\in\Ab.
\end{cases}
\end{equation}
From the second condition in (\ref{1ordine}) it easily follows
$\theta_i=\theta_j$, $i,j=1,...,q_a$, and hence if $\Aa$ is not
empty then it contains only copies of the same angle $\theta$ and
hence $\Aa$ is finite.

Let us consider the third condition in (\ref{1ordine}); for $\eta_i,\eta_j\in\Ab$ it holds
$$
\la\, b (\cos\eta_i-\cos\eta_j)(\cos\eta_i+\cos\eta_j)=\cos\eta_i-\cos\eta_j,
$$
which implies either $\eta_i=\eta_j$ or $\eta_i\neq\eta_j$ with
\begin{equation}\label{cosxcosy}
\cos\eta_i + \cos\eta_j = \frac 1{b\la}.
\end{equation}
Hence $\Ab$ contains at most two different angles; let us call them $x,y$ and assume $x>y$.
This implies that also $\Ab$ is a finite set.

By the second order optimality conditions we have that  for every
$d\in\rn$ which belongs to the critical cone associated to
$\bar{X}$, that is such that $d$ verifies
\begin{equation}\label{criticalcone}
\begin{cases}
\langle D\jla{\la}(\bar X); d \rangle &\le 0,\\
\langle Dg_k(\bar X); d\rangle &\le 0, \qquad\text{ for }k=1,...,p\\
\langle Dh(\bar X); d\rangle &=0,
\end{cases}
\end{equation}
it holds
\begin{equation}\label{D2jla>0}
\langle D^2\jla{\la}(\bar X) d, d \rangle \ge 0,
\end{equation}
where $D^2\jla{\la}$ is the diagonal matrix
\begin{equation}\label{D2jla}
[D^2\jla{\la}(X)]_{ii}=
\begin{cases}
2\frac{b^2}{a^2}\sqrt{b^2-a^2}\, (a\,\la-2)  &\qquad\text{if }i=1,...,p\\
2a(a\,\la-2)\, \frac{\sin\theta}{\cos^3\theta}&\qquad\text{if }i=p+1,...,p+q_a\\
2b(-b\,\la\sin 2\eta_j + \sin\eta_j)        &\qquad\text{if }i=N-q_b+1,...,N.\\
\end{cases}
\end{equation}

Assume $q_a=|\Aa|\ge 2$ and let $d$ be a vector in the critical
cone with $d_i=0$ if $i=1,...,p, N-q_b+1,...,N$ (that is $d$ has non
null components only corresponding to the elements of the class
$\Aa$). Hence
$$
\langle D^2\jla{\la}(\bar X)d;d \rangle=2a^2\left(\la-\frac 2a\right)\frac{\sin\theta}{\cos^3\theta}\sum_{i=p+1}^{p+q_a} d_i^2,
$$
which is negative and hence contradicts (\ref{D2jla>0}).
This proves {\bf\ref{thetatheta}}.

Assume there exist $\eta_j=\eta_k=z\in\Ab$ and consider $d\in\rn$ such that $d_j=-d_k$ and $d_i=0$ for $i\neq j,k$.
Hence $d$ belongs to the critical cone (\ref{criticalcone}) and hence (\ref{D2jla>0}) holds, that is
$$
2b d_j^2\sin z(1-2b\,\la\cos z)\ge 0,
$$
which entails $\cos z\le 1/2b\la$.
Analogously, assume $\eta_j=y$, $\eta_k=x$ with $\cos x+\cos y=1/b\la$  by (\ref{cosxcosy}); consider the same $d$ as before.
Condition (\ref{D2jla>0}) gives
$$
2b\,d_k^2(\sin y-\sin x)(1-2b\,\la\cos y)\ge 0,
$$
which implies that if $x>y$ then $\cos y \ge 1/2b\la$ (and hence by (\ref{cosxcosy}) $\cos x\le 1/2b\la$).

Assume $\Ab$ contains the set $\{x,y,y\}$ with $x>y$; then it holds $\cos y= 1/2b\la$ which implies $x=y$ by (\ref{cosxcosy}).
Hence the thesis holds true.
\end{proof}
\begin{remark}\label{rmkoptimality}
The Hessian matrix $D^2\jla{\la}$ is the diagonal matrix given in (\ref{D2jla}).
Since the critical (tangent) cone is here an hyperplane, three situations can occur:
\begin{itemize}
  \item all the eigenvalues of $D^2 \jla{\la}$ are non negative and the second order optimality condition is automatically fulfilled;
  \item there exist \emph{at least} two negative eigenvalues and the quadratic form cannot be non negative on a hyperplane, thus the second order
  optimality condition is not satisfied;
  \item there exists one and only one negative eigenvalue.
 In this case, as explained in \cite[Corollary 4.6]{HPR}, the quadratic
 form with eigenvalues $\la_1<0<\la_2\leq \la_3\leq \ldots\la_N $ will be non
 negative on the hyperplane $H=(x_1,x_2,\ldots x_N)^\bot$ if and only if
\begin{equation}\label{ineqHPR}
\sum_{i=1}^N \frac{x_i^2}{\lambda_i} \le 0.
\end{equation}
\end{itemize}

In our situation, to each angle $\theta\in \Aa$ or $y\in \Ab$ corresponds a negative eigenvalue of $D^2\jla{\la}$.
This is the reason why we cannot have more than one of such angles.
Moreover, as soon as one of these angles $\theta\in \Aa$ or $y\in \Ab$ exists, the inequality (\ref{ineqHPR}) gives an information which will be useful
in the sequel, see Section \ref{SecNumbersOfSides}.

As pointed out in {\bf \ref{xxy}.} of Theorem \ref{lemmaAaAb} if there exist two different angles $x>y$ then $\cos x\le 1/2\la b$.
More precisely this holds true also if the class $\Ab$ is composed only by copies of a same angle $x$.
Indeed if $\Ab\supseteq\{x,x\}$, the eigenvalue of $D^2\jla{\la}(\ola)$  associated to $x$ has to be non negative, that is $2b\sin x(1-2b\,\la\cos x)\ge 0$, which gives $\cos x \le 1/2\la b$.
\end{remark}

As already noticed in the proof of Theorem \ref{lemmaAaAb}, the following holds.
\begin{corollary}\label{Nfinite}
Let $\ola\in \CC$ be an optimal set such that $\bd\oo=\cup_{i\in I}L_i$, with $L_i\in\lab\cup\laa\cup\lbb$.
Then $I$ is a finite set of indices and hence for $1/2b<\la<2/a$ the set $\ola$ is a polygon.
\end{corollary}
This implies that for $1/2b<\la<2/a$ the minimum Problem (\ref{PB}) can be explicitly rewritten as a function of the central angles of the polygon. 
In particular if $\oo$ is an $N$-gone, we define $X$ its vector of central angles such that
$X=(\xi_0,...,\xi_0,\theta,x,...,x,y)$ that is $x_i$ corresponds to the elements of the classes $\Aab,\Aa,\Ab$ for $i=1,...,p$, $i=N-q_b+1,...,N$, respectively; where $|\Aab|=p$, $|\Aa|=1$, $|\Ab|=q_b$ with $p+1+q_b=N$.
Recalling (\ref{JlaCC}) we have
$$
\min_{\oo\in\C}\jla{\la}(\oo)=\min_{X\in\AAA}\jla{\la}(X),
$$
where
$$
\AAA=\{X\in\rn\quad\text{ such that }\quad \ \sum_{i=1}^N x_i=\pi, \quad x_i<\xi_0\ \text{ for }\ i=p+1,...,N\},
$$
and
\begin{eqnarray}
\jla{\la}(X)&=&\la\Big( \sum_{i=1}^{p+1} a^2\tan x_i +b^2
\sum_{j=N-q_b}^N\sin x_j\cos x_j \Big) - 2\Big( \sum_{x_i=1}^{p+1}a\tan\xi_0 +b\sum_{x_j=N-q_b}^N\sin x_j \Big). \nonumber
\end{eqnarray}
%
%
\subsection{Optimal shape for extremal values of $\la$}
We here analyse the case of extremal values of $\la$. 
In the limit cases $\la=0$ or $\la=+\infty$ the solution to (\ref{PB}) is evident to be the exterior ball $D_b$ and the interior one $D_a$, respectively. It is
in fact the same also for values of $\la$ near to these limit cases.
\begin{theorem}\label{teoDaDb}
Let $\ola$ be a minimizer to (\ref{PB});
\begin{enumerate}
\item\label{Db} if $\la\le1/2b$ then $\ola$ is unique and $\ola=D_b$;
\item\label{Da} if $\la >2/a$ then $\ola$ is unique and $\ola=D_a$.
\end{enumerate}
\end{theorem}

In order to prove this result some preliminary steps are needed.
They are collected in the following lemmas.

\begin{lemma}\label{Lavuoto}
For every $\la\le 2/(a+b)$, $\ola$ does not contain tangent sides to $D_a$ which are not chord of $D_b$.
\end{lemma}
\begin{lemma}\label{Lbvuoto}
For every $\la\ge 1/a$, $\ola$ does not contain chords of $D_b$ which are not tangent to $D_a$.
\end{lemma}

\begin{proof}[Proof of Theorem \ref{teoDaDb}]
This proof is in fact analogous and at the same time opposite to the
proof of Theorem \ref{teopolyg}. Indeed we here consider the same
constructions as before, to  prove the exact complement: for $\la\le
1/2b$ and $\la>2/a$, the set $\ola$ does not contains segments.

\medskip

Proof of part {\bf \ref{Db}}.

\noindent As $\la<2/(a+b)$, by Theorem \ref{teopolyg}, Lemma
\ref{Lavuoto} and Theorem \ref{teolaalablb} we have that if $\bd\oo$
contains a segment, then it is a chord of $D_b$. Let $AB$ be one of
these chords, $A=b\ee^{i\theta_A}$, $B=b\ee^{i\theta_B}$. We define
$\oet$ starting from $\oo$ and substituting the chord $AB$ with the
corresponding arc on $D_b$; with $\eta= (\theta_A-\theta_B)/2$.
%
%
We compare $\jla{\la}(\oo)$ with $\jla{\la}(\oet)$ getting
$$
\jla{\la}(\oet)-\jla{\la}(\oo)= \frac b2 (\eta-\sin\eta\cos\eta)\left(\la 2 b- 4\frac{\eta-\sin\eta}{\eta-\sin\eta\cos\eta}\right),
$$
which is negative as $\la 2 b\le 1$ and
$$
\frac{\eta-\sin\eta}{\eta-\sin\eta\cos\eta}>\frac 14,\qquad\text{ for every }\eta>0.
$$
Hence $\bd \ola$ does not contain chords of $D_b$, which implies
$\ola=D_b$ since by step 2 in Theorem \ref{teopolyg}, $\bd\ola$ does
not contain neither arcs of $D_a$.

\medskip

Proof of part {\bf \ref{Da}}.

\noindent As $\la>1/a$, by  Lemma \ref{Lbvuoto} and Theorem \ref{teolaalablb}  we have that $\ola$ can be composed only by arcs of $\bd D_a$ and tangent segments to $D_a$; let $AB$, $BC$, with $A,C\in\bd D_a$, $A=a\ee^{i \theta_A}$, $C=a\ee^{i \theta_C}$, be some of them. 
Let $\eta$ be such that $\tan\eta=\overline{AB}/a=\overline{BC}/a$ and let us consider the set $\oet$ obtained from $\oo$ substituting the segments $\overline{AB}, \overline{BC}$ by the corresponding arc of $D_a$, ${AC}$.
Computing $\jla{\la}(\oo)$, and $\jla{\la}(\oet)$ we get
$$
\jla{\la}(\oo)-\jla{\la}(\oet)=a^2(\tan\eta-\eta)(\la-\frac 2a),
$$
which is positive and hence $\ola$ cannot contain tangent segments to $D_a$.
This entails that $\ola=D_a$.
\end{proof}

We now give the proof of Lemmas \ref{Lavuoto} and \ref{Lbvuoto}.
Notice that we here use non-local perturbations of $\oo$.
\begin{proof}[Proof of Lemma \ref{Lavuoto}]
Let $\oo$ be a set in the class $\C$ with $x\in\Aa$; let $PQ$, $QR$ be the corresponding tangent sides to $D_a$.
Notice that we can assume $R,P\in\bd D_b$ as by Theorem \ref{lemmaAaAb} there exists at most one angle in the class $\Aa$.

\begin{center}
\begin{figure}[h]
\begin{tikzpicture}[x=0.40mm,y=0.40mm]
\draw[thin] (0,0) circle (50);
\draw (-46,-20) node [right] {$D_b$};
\draw (-30,0) node [right] {$D_a$};
\draw[thin] (0,0) circle (30);
\draw[thick,-|] (-40,30) node [left] {$R$} -- (12.42640687,30) node [above] {$Q$};   
\draw[thick] (12.42640687,30) -- (49.49747468, -7.07106781) node [right] {$P$}; 
\draw[very thin] (-40,30) -- (40,30) node [right] {$M$};
\draw[very thin] (40,30) -- (49.49747468, -7.07106781);  
\draw[very thin] (0,0) -- (12.42640687,30);
\draw[very thin] (0,0) -- (21.21320344,21.21320344);
\draw[very thin] (45:9) arc (45:68:9) node[above right]{$x$};
\draw[very thin] (45:11) arc (45:68:11);
\draw (-30,30) node [below] {$\oo$};
\draw[->, very thin] (35,47) node[right]{$\ot$}--(30,32);
\draw (0,-60) node{(a)};
\begin{scope}[xshift=200]
\draw[thin] (0,0) circle (50);
\draw (-46,-20) node [right] {$D_b$};
\draw (-30,0) node [right] {$D_a$};
\draw[thin] (0,0) circle (30);
\draw[thick,-|] (-40,30) node [left] {$R$} -- (12.42640687,30) node [above] {$Q$};   
\draw[very thin] (12.42640687,30) -- (49.49747468, -7.07106781) node [right] {$P$}; 
\draw[thick,|-] (0,30) node[above]{$H$} -- (40,30) node [right] {$M$};
\draw[thick] (40,30) -- (49.49747468, -7.07106781);  
\draw (44.5,11.5) node[right]{N};
\draw[thick] (42,11.5)--(47,11.5);
\draw[very thin] (0,0) -- (12.42640687,30);
\draw[very thin] (0,0) -- (21.21320344,21.21320344);
\draw[very thin] (45:9) arc (45:68:9) node[above right]{$x$};
\draw[very thin] (45:11) arc (45:68:11);
\draw (37,-4) node {$\ot$};
\draw[->, very thin] (35,47) node[right]{$\oo$}--(30,32);
\draw (0,-60) node{(b)};
\end{scope}
\end{tikzpicture}
\caption{For $\la\le 2/(a+b)$, $\laa=\emptyset$; for $\la\ge 1/a$, $\lbb=\emptyset$}\label{Lbvuoto1}
\end{figure}
\end{center}

Consider a set $\ot$ as in Figure \ref{Lbvuoto1} (a), obtained from $\oo$ by moving the point $Q$ along the line $RQ$, up to the point $M$ on the boundary of $D_b$. 
Hence, 
$$
\jla{\la}(\ot)-\jla{\la}(\oo)= \frac{\la}2 \overline{QM}\,\overline{QP}\sin 2x -(\overline{QM}+\overline{PM}-\overline{QP}).
$$ 
As $\overline{QM}=\sqrt{b^2-a^2}-a\tan x$, and $\overline{QP}=\sqrt{b^2-a^2}+a\tan x$, we obtain $\overline{PM}=2b\sin x$ and hence
$$
\jla{\la}(\ot)-\jla{\la}(\oo)=\frac{\sin x}{\cos x}(b^2\cos x-a^2) \left(\la-\frac 2{b\cos x+a}\right),
$$
which is negative since $\la\le 2/(a+b) <2/(b\cos x+a)$.
This shows that if $\la\le 2/(a+b)$ then the class $\laa$ has to be empty.
\end{proof}

\begin{proof}[Proof of Lemma \ref{Lbvuoto}]
Assume that $\bd\oo$ contains a chord $MP$, with $M,P\in\bd D_b$,
with $MP$ not tangent to $D_a$. By Lemma \ref{Lb<1} we can assume
$\lbb=\{MN\}$, where $N$ is the  middle point of $MP$; then there
exists a side $MR$ which touches $D_a$ at a point $H\in\bd\oo\cap\bd
D_a$.

Consider the set $\ot$ obtained from $\oo$ by moving the point $M$
along the line $HM$ up to the point $Q$ such that $QP$ is tangent to
$D_a$ (see Figure \ref{Lbvuoto1} (b)). As in the proof of Lemma
\ref{Lavuoto} we get
$$
\jla{\la}(\oo)-\jla{\la}(\ot)=\frac{\sin x}{\cos x}(b^2\cos x-a^2) \left(\la-\frac 2{b\cos x+a}\right),
$$
which is positive for $\la\ge 1/a$ as $b\cos x+a>2a$.
This shows that if $\la\ge 1/a$ then $\lbb$ has to be empty.
\end{proof}

\section{Further characterizations}\label{SecNumbersOfSides}
This section is devoted to a more precise analysis of optimal sets,
in particular regarding the total number of sides, and further
properties of the classes $\lab,\laa,\lbb$. These results are useful
if one wants to describe the optimal sets for a given value of $a,b$
as shown in Section \ref{examples}.

\subsection{Analysis of large values of $\la$}
In this section we give a complete characterization of optimal sets
for sufficiently large values of $\la$. In particular in Theorem
\ref{teopxi0qx} we give the exact  number of sides of an optimal
polygon together with a description of its classes of central angles
for $1/b\le\la<2/a$.
\begin{theorem}\label{teopxi0qx}
For $1/b\le\la<2/a$ optimal sets $\ola$ are polygons in the class $\CC$ with a minimum number of sides. 
In particular let $p_0$ be the largest integer such that  $p_0\xi_0\le\pi$, where $\xi_0$ is defined as in (\ref{xi0}); then $|\Aab|=p_0$ and either $\Ab$ is empty or so is $\Aa$.

More precisely let $x=\pi-p_0\xi_0$; if $x\neq 0$ then $\ola$ is
inscribed into $D_b$ for $1/b\leq \la\le 2/({b\cos x+a})$ while it
is circumscribed to $D_a$ for $2/({b\cos x+a})\le \la<2/a$ and $\ola$ has either $p_0$ or $p_0+1$ sides.
\end{theorem}

Before giving the proof we present some preliminary results, namely Lemma \ref{Aabnotempty} and Lemma \ref{Lb<1}.
\begin{lemma}\label{Aabnotempty}
Let $\ola$ be an optimal set with $\la\ge 1/b$. Then its class $\Aab$ is not empty.
\end{lemma}
\begin{proof}
Let $p=|\Aab|, q=|\Ab|$.
Assume $q\ge 2$ hence we have: either $\Ab=\{x,y\}$ or $\Ab$ contains at least two copies of the same angle $x$.

As pointed out in Remark \ref{rmkoptimality}, in both cases optimality conditions (\ref{1ordine}), (\ref{criticalcone}), (\ref{D2jla>0}) imply
\begin{equation}\label{xi1}
\cos x\le \frac 1{2\la b}.
\end{equation}
that is $x\ge x_0$, where $\cos x_0=1/2\la b$.
Hence $p\xi_0+(q-1) x\ge p\xi_0+(q-1)x_0$ which entails
\begin{equation*}
p\xi_0+(q-1)x_0<\pi.
\end{equation*}

If $p=0$ then $q\ge 3$ as $\Aab$ empty implies $\Aa$ is empty as well. 
In particular if $q\ge 4$ then previous argument implies $x_0<\pi/3$ which contradicts the fact that $\cos x_0=1/2\la b\le1/2$. 
If $q=3$ then $\ola$ is an isosceles triangle identified by its central angles as $\{x, x, \pi-2x\}$ with $x\in[\frac {\pi}3,\xi_0]$. 
By direct computations it turns out that the functional $\jla{\la}$ is  monotone decreasing as function of $x$, for every $\la\ge \frac 1b$ and hence the optimal isosceles triangle is determined for $x=\xi_0$, which contradicts the fact that $\Aab$ is empty.
\end{proof}

\begin{lemma}\label{Lb<1}
For $\la\ge \min\{1/2a,1/b\}$ any optimal set $\ola$ has $|\lbb|\le 1$.
\end{lemma}
\begin{proof}
Assume $|\lbb|=|\Ab|\ge 2$, then as shown in the proof of Lemma \ref{Aabnotempty} the optimality conditions (\ref{criticalcone}), (\ref{D2jla>0}) implies (\ref{xi1}), which gives $\la< 1/2a$ as by construction $\cos x >a/b$. 
Hence if $\la\ge 1/2a$ it holds $|\Ab|\le 1$.

\medskip

Assume now $\la\ge1/b$; as we have already proved the thesis for $\la\ge 1/2a$, it is sufficient to consider the case $b> 2a$.
Let $p=|\Aab|$, $q=|\Ab|$ and $N$ be the total number of sides of $\ola$.
as $b>2a$, it holds $\xi_0>\pi/3$.
%
%
Assume $\Ab$ contains two different angles $x,y$ with  $x>y$, hence by Theorem \ref{lemmaAaAb} $\Ab$ contains $(q-1)$ copies of $x$ and one copy of $y$ with $q\ge 2$ and it holds $p\xi_0+(q-1)x+y\le \pi$ (where equality holds if there does not exist an angle $\theta\in\Aa$). 
Notice that as $\cos x\le 1/(2\la\,b)$ and $\la>1/b$, it holds $x>\pi/3$. 
Moreover, by construction, $\xi_0>x$, which gives $\pi>(p+q-1)\pi/3$ which implies $p+q<4$ that is $p+q\le 3$ and hence the only possibility is $p+q=3$ either with $\Aa$ empty, or with $\Aa=\{\theta\}$ for some $\theta$.

The case $\Aa$ not empty cannot be optimal as it implies $q=2, p=1$ and by translation we can easily obtain a new domain $\widetilde\oo$ whose sides do not belong to $\lab\cup\laa\cup\lbb$ such that the value of $\jla{\la}$ is unchanged. As $\widetilde \oo$ cannot be optimal, so is not $\oo$. 
In the case $\Aa=\emptyset$ we only have two candidates: the triangles $T'$ and $T''$ determined by their sets of angles as $\{\xi_0,x,y\}$ and $\{z,z,u\}$, respectively. 
By a direct computation we obtain that neither $T'$ nor $T''$ are optimal; in fact $\jla{\la}$ can be seen as a function of $x,z$, respectively, which decreases for $x,z\in(0,\xi_0)$ for $\la\ge b/(b^2+ab-2a^2)$, which is the case for $\la\ge 1/b$ (or $\la\ge1/2a$).
%
%
Hence,
$$
\jla{\la}(T'), \jla{\la}(T'')>\jla{\la}(T),
$$
where $T$ is the triangle determined by the angles $\{\xi_0,\xi_0,\pi-2\xi_0\}$.

Assume now that the class $\Ab$ only contains copies of a same angle $x$, so that $p\xi_0+qx\le\pi$.
We want to prove that $q\le 1$.
Indeed if $q\ge 2$ then the optimality conditions (\ref{criticalcone}), (\ref{D2jla>0}) implies (\ref{xi1}) (see Remark \ref{rmkoptimality}).
In particular for $\la\le 1/b$ we obtain $\cos x\le 1/2$ that is $x\ge\pi/3$ and this gives $q\le 2$ and then $q=2$.
We then have
$$
\pi\ge p\xi_0+2x>(p+2)\frac{\pi}3,
$$
which entails $p<1$ that is $p=0$ and hence $N=q=2$, which is absurd.

Hence $\ola$ is a triangle with the max number of sides which are at the same time tangent to $D_a$ and chord of $D_b$ and hence $|\lbb|\le 1$.
\end{proof}

We finally present the proof of Theorem \ref{teopxi0qx}.
\begin{proof}[Proof of Theorem \ref{teopxi0qx}]
We are going to prove that $|\Aa|\cdot|\Ab|=0$; we split the proof in the cases $b\lessgtr2a$.

Case $b\ge 2a$. Assume both $\Aa$ and $\Ab$ not empty. By the proof
of Lemma \ref{Lb<1} we have that $\ola$ is necessarily a triangle
hence we have $\Aab=\{\xi_0\}$, $\Aa=\{\theta\}$, $\Ab=\{x\}$, which
is not optimal as already noticed in the above proof since it can be translated. Hence either
$\Aa$ is empty and we get the triangles $T'$ with $\Ab=\{x\}$, or so
is $\Ab$ and we obtain $T''$ with $\Aa=\{\theta\}$, respectively
with $x=\theta=\pi-2\xi_0$.
\begin{center}
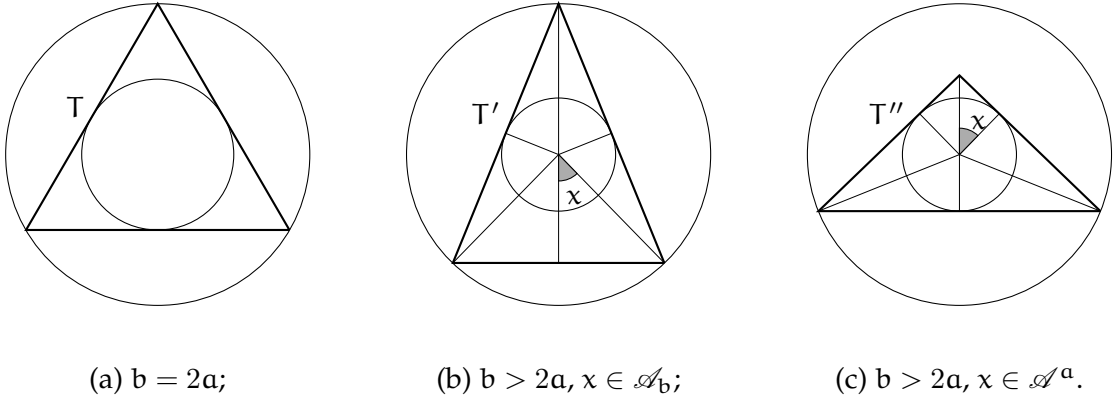
\begin{figure}[h]
\begin{tikzpicture}[x=0.50mm,y=0.50mm]
\begin{scope}[xshift=-150]
\draw[thin] (0,0) circle (40);
\draw[thin] (0,0) circle (20);
\coordinate (A) at (34.64,-20); \coordinate (B) at (-34.64,-20);
\draw[thick]  (A)--(0,40)--(B) --(A);
\draw (150:25) node{$T$};
%
%
\draw (0,-60) node{(a) $b=2a$;};
\end{scope}
\draw[thin] (0,0) circle (40);
\draw[thin] (0,0) circle (15);
\draw[thick] (-45.94:40)--(90:40)--(225.94:40)--(-45.94:40);
\draw (150:22) node{$T'$};
\draw[very thin] (0,0)--(-45.94:40);\draw[very thin] (0,0)--(90:40);
\draw[very thin] (0,0)--(225.94:40);\draw[very thin] (0,0)--(0,-29);\draw[very thin] (0,0)--(22.03:15);\draw[very thin] (0,0)--(157.97:15);
\draw[fill=gray!65] (0,0)--(270:7)  arc(270:314.06:7)-- cycle;
\draw (290:12)  node {$\tiny x$};
\draw (0,-60) node{(b) $b>2a$, $x\in\Ab$;};
\begin{scope}[xshift=150]
\draw[thin] (0,0) circle (40);
\draw[thin] (0,0) circle (15);
\draw[thick] (-22.03:40)--(202.03:40)--(90:21)--(-22.03:40);
\draw (150:22) node{$T''$};
\draw[very thin] (0,0)--(-22.03:40);\draw[very thin] (0,0)--(45.94:15);
\draw[very thin] (0,0)--(90:21);\draw[very thin] (0,0)--(134.06:15);\draw[very thin] (0,0)--(202.03:40);\draw[very thin] (0,0)--(-90:15);
\draw[fill=gray!65] (0,0) -- (45.94:7) arc (45.94:90:7)-- cycle;
\draw (60:10)  node {$\tiny x$};
\draw (0,-60) node{(c) $b>2a$, $x\in\Aa$.};
\end{scope}
\end{tikzpicture}
\caption{The triangles $T$, $T'$, $T''$, respectively.}\label{triangles}
\end{figure}
\end{center}
Otherwise both $\Aa$ and $\Ab$ are empty, hence $T$ is the regular
triangle of side $\sqrt{b^2-a^2}$. By explicit computation we obtain
that $\ola=T$ if $b=2a$ (notice  that in this case $T$ is the only
triangle which belongs to the class $\CC$), while for $b>2a$ we have
$$
\ola=
\begin{cases}
T' &\quad\text{for }\quad \frac 1b<\la\le \frac{2b}{(b-a)(b+2a)},\\
T''&\quad\text{for }\quad \frac{2b}{(b-a)(b+2a)}\le\la<\frac 2a.
\end{cases}
$$

Case $b< 2a$.
If both $\Aa$ and $\Ab$ are not empty, by Lemma \ref{Lb<1} and Theorem \ref{lemmaAaAb} it holds $\Aa=\{\theta\}$, $\Ab=\{x\}$.
Let $p=|\Aab|$.
By construction $p+2\ge 4$ hence
$$
\pi=p\xi_0+\theta+x>(p+1)x\ge 3x,
$$
which gives $x<\pi/3$, and hence $\cos x>1/2$. Consider the second
order optimality conditions (\ref{criticalcone}), (\ref{D2jla>0})
and let $d=(0,...,0,-k,k)\in\rn$ be a vector in the critical cone,
where the last two components correspond to the element of $\Aa$ and
$\Ab$ respectively. Computing the second derivatives of $\jla{\la}$
we get
$$
\langle D^2\jla{\la}(\ola)d;d\rangle=k^2\Big( 2a(a\la-2)\frac{\sin\theta}{\cos^3\theta} +2b\sin x(1-2b\la\cos x) \Big),
$$
which is negative as we showed that $\cos x >1/2$.
This is a contradiction.

Hence $\ola$ is either inscribed into $D_b$ or circumscribed to
$D_a$ with at most one side which does not belong to $\lab$. This
means that $\ola$ is a polygon composed by the maximum number  of
segment in $\lab$ which are completed by a last segment determined
by a central angle  which belongs either to $\Aa$ or to $\Ab$. More
precisely, $\ola$ can be represented by its central angles as the
set of $p$ copies of $\xi_0\in\Aab$ with a last angle $x=\pi-p\xi_0$
such that $x<\xi_0$. Denote by $\ola^a$ the set corresponding to
$x\in\Aa$ and $\ola^b$ that corresponding to $b\in\Ab$. 
\begin{center}
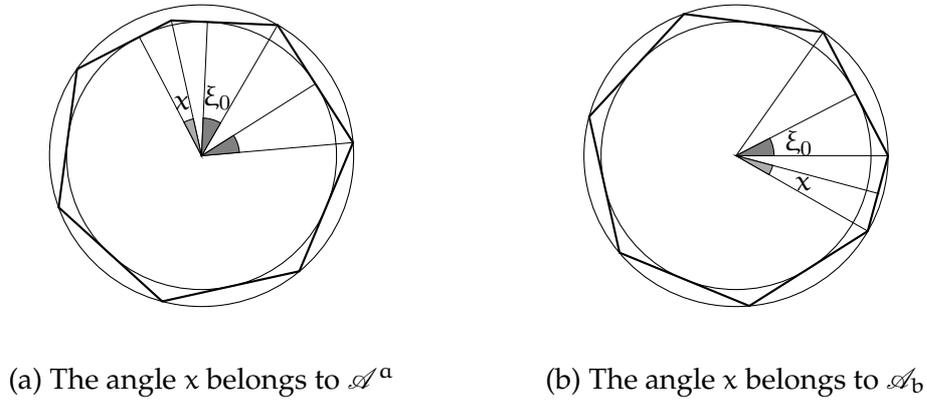
\begin{figure}[h]
\begin{tikzpicture}[x=0.50mm,y=0.50mm]
\draw[thin] (0,0) circle (40);
\draw[thin] (0,0) circle (35.5);
\draw[thick] (0:40)--(55:40)--(110:40)--(165:40)--(220:40)--(275:40)--(330:40)--(360:40);
\draw[very thin] (0,0)--(330:40);\draw[very thin] (0,0)--(0:40);
\draw[very thin] (0,0)--(345:39);\draw[very thin] (0,0)--(55:40); \draw[very thin] (0,0)--(27.5:35.5);
\draw[fill=gray] (0,0) -- (0:10) arc (0:27.5:10)-- cycle;
\draw (10,3.5)  node[right] {$\tiny\xi_0$};
\draw[fill=gray!65] (330:0) -- (330:10) arc (330:345:10)-- cycle;
\draw (330:15)  node[right] {$\tiny x$};
\draw (0,-60) node{(b) The angle $x$ belongs to $\Ab$};
\begin{scope}[xshift=-200]
\draw[thin] (0,0) circle (40);
\draw[thin] (0,0) circle (35.5);
\draw[thick] (145:40)--(200:40)--(255:40)--(310:40)--(5:40)--(60:40)--(102.5:36.75)--(145:40);
\draw[very thin] (0,0)--(102.5:36.75);\draw[very thin] (0,0)--(117.5:35.5);
\draw[very thin] (0,0)--(87.5:35.5);\draw[very thin] (0,0)--(60:40); \draw[very thin] (0,0)--(32.5:35.5);\draw[very thin] (0,0)--(5:40);
\draw[fill=gray!65] (0,0) -- (102.5:10) arc (102.5:117.5:10)-- cycle;
\draw (110:15)  node {$\tiny x$};
\draw[fill=gray] (87.5:0) -- (87.5:10) arc (87.5:60:10)-- cycle;
\draw[fill=gray] (5:0) -- (5:10) arc (5:32.5:10)-- cycle;
\draw (75:16)  node {$\tiny \xi_0$};
\draw (0,-60) node{(a) The angle $x$ belongs to $\Aa$};
\end{scope}
\end{tikzpicture}
\caption{The sets $\ola^a$ and $\ola^b$ in the proof of Theorem \ref{teopxi0qx}.}\label{olaAB}
\end{figure}
\end{center}
By a direct computation we get
$$
\jla{\la}(\ola^a)-\jla{\la}(\ola^b)=\frac{\sin x}{\cos x}(a-b\cos x)\big( \la(a+b\cos x) -2\big),
$$
and hence
$$
\ola=
\begin{cases}
\ola^b &\quad\text{for }\quad\frac 1b<\la\le \frac 2{b\cos x+a},\\
\ola^a &\quad\text{for }\quad\frac 2{b\cos x+a}\le \la<\frac 2a.
\end{cases}
$$
\end{proof}

\subsection{Analysis of small values of $\la$}
By Lemma \ref{Lavuoto} and Corollary \ref{corollaryPolyg} for
$1/2b<\la<2/(a+b)$ an optimal set is a polygon inscribed into $D_b$
with possible tangent sides to $D_a$. In particular by Lemma
\ref{Lb<1} there exists at least one chord which is tangent to the
interior ball, for $\la\ge\min\{1/2a,1/b\}$

The following proposition expresses the fact that if $\la$ is
sufficiently small (but sufficiently large to have a polygonal
solution), then optimal sets are strictly inscribed into $D_b$.
\begin{proposition}\label{labvuoto}
Let $\ola$ be an optimal set, with $1/2b<\la<1/(a+b)$. Then the classes $\lab$ and $\laa$ are empty.
\end{proposition}

\begin{proof}
Let $\oo$ be a polygon inscribed into $D_b$; assume that there
exists a chord $PQ$ of $D_b$ which is tangent to $D_a$, that is
$PM\in\lab$ where $M$ is the middle point of $PQ$. As shown in
Figure \ref{trapezio} we consider the set $\oo^\ep$ obtained as a
perturbation of $\oo$ constructing two new points $P_\ep,Q_\ep$ on
$\bd D_b$, such that $\overline{PP_\ep}=\overline{QQ_\ep}=\ep$ (and
hence $Q_\ep P_\ep$ is parallel to $PQ$) with  $P_\ep
Q_\ep\cap\oo=\emptyset$. Let us denote by $\eta=\eta(\ep)$ the angle
between $PQ$ and $P_\ep P$.
\begin{center}
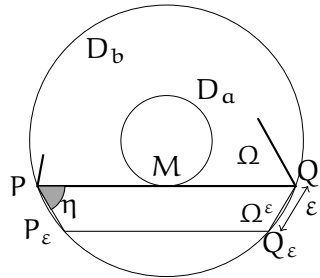
\begin{figure}[h]
\begin{tikzpicture}[x=0.6mm,y=0.6mm]
\draw[very thin] (0,0) circle (30);
\draw (-20,20) node [right] {$D_b$};
\draw (11,11) node {$D_a$};
\draw[very thin] (0,0) circle (10);
\draw (18,-3) node{\small$\oo$};
\draw[thick] (-27,-3)--(-28.28427125,-10) node [left] {$P$} -- (28.28427125,-10)--(20,5);   
\draw (31,-7) node {$Q$};
\draw (-28.28427125,-10)--(-22.360679775,-20) node [left]{$P_\ep$}--(22.360679775,-20)--(28.28427125,-10);
\draw (25,-23) node{$Q_\ep$};
\draw (20,-16) node{\small$\oo^\ep$};
\draw (0,-10) node[above]{$M$};
\draw[fill=gray!70] (-28.284,-10) -- (-22.284,-10) arc (0:-60:6) node [right]{$\eta$}-- cycle;
\draw[<->] (25,-20) -- (31,-10);
\draw (28,-15) node[right]{$\ep$};
\end{tikzpicture}
\caption{The construction of $\oo^\ep$: for $\frac 1{2b}<\la<\frac 1{a+b}$, $\Aab=\emptyset$.}\label{trapezio}
\end{figure}
\end{center}
Again we want to show that in fact $\jla{\la}(\oo^\ep)<\jla{\la}(\oo)$.
Consider
$$
\jla{\la}(\oo^\ep)-\jla{\la}(\oo) = \ep\sin\eta (2\sqrt{b^2-a^2}-\ep\cos\eta)\Big(\la-\frac{2\tan(\textstyle{\eta/2})}{2\sqrt{b^2-a^2}-\ep\cos\eta}\Big);
$$
and notice that
$$
\lim_{\ep\to 0}\frac{2\tan(\textstyle{\eta/2})}{2\sqrt{b^2-a^2}-\ep\cos\eta}=\frac 1{a+b},
$$
since $\lim_{\ep\to 0}\eta(\ep)=\xi_0$.
Hence, as $\la< 1/(a+b)$, there exists $\ep>0$ (and hence $\eta>0$) such that $\jla{\la}(\oo^\ep)-\jla{\la}(\oo)<0$.
\end{proof}

As already noticed for small values of $\la$ optimal polygons are
inscribed into $D_b$. In particular for $1/2b<\la<1/b$ either $\ola$
contains tangent sides  to $D_a$, or it is either regular or
``quasi-regular'', where quasi-regular means that it has all the
sides of equal length, except one. It would be interesting to
investigate when each of the cases arrives.

Now let us consider the case of quasi-regular polygons.
Notice that not for every values of $\la, a,b, N$ an optimal quasi-regular $N$-gone can be constructed in the class $\CC$.
In particular some estimates of the possible number of sides of an optimal polygon holds.
\begin{proposition}\label{pq}
Let $\ola$ be an optimal polygon, with $1/2b<\la\le 1/b$ and let $p=|\Aab|$, $q=|\Ab|$.
It holds
\begin{equation}\label{pq1}
p_0+1-p\le q\le \frac {\pi}{x_0}-p\frac{\xi_0}{x_0}+1,
\end{equation}
where $\xi_0$ is defined in (\ref{xi0}), $\cos x_0=\frac 1{2\la b}$ and $p_0=[\textstyle\frac{\pi}{\xi_0}]$.

In particular if $\Ab\supseteq\{x,y\}$ it also holds
$$
\frac{\pi-p\xi_0+\sqrt{(\pi-p\xi_0)^2-{\textstyle\frac 92} x_1}}{2x_1}+1\le q\le \frac{\pi}{x_1}-p\frac{\xi_0}{x_1}+1.
$$
\end{proposition}
\begin{proof}
Notice that $p_0$ represents the maximum number of copies of the angle $\xi_0$ that a polygon in the class $\CC$ can have as central angle.
That is $p_0\xi_0\le \pi<(p_0+1)\xi_0$.
Hence the minimum number of sides is always at least $p_0$, and equality holds only in the case $p_0=\pi/\xi_0$.
In the general case $\pi/\xi_0\not\in\N$, it holds in fact $N\ge p_0+1$, that is
$$
q\ge p_0+1-p.
$$
In what follows we assume $p_0<\pi/\xi_0$, in order to treat a more general situation.

Notice that, by optimality conditions, if $\Ab$ contains a couple of
equal  angles $\{x,x\}$ or a couple of angles $\{x,y\}$, it holds $\cos x\le
1/2\la b$ (see Theorem \ref{lemmaAaAb} for the case
$\{x,y\}\subseteq\Ab$ and Remark \ref{rmkoptimality} for the case
$\{x,x\}\subseteq\Ab$). Hence if $q\ge 2$, and $\Ab$ has at least
$(q-1)$ copies of an angle $x$, it holds $x_0\le x\le \xi_0$ with
$\cos x_0=1/2\la b$.

Assume that $\Ab$ only contains $q$ copies of the same angle $x$, such that $p\xi_0+q x=\pi$; we have
$$
q\le \frac{\pi}{x_0}-p\frac{\xi_0}{x_0}.
$$
If $\Ab$ contains a couple of angles $\{x,y\}$, that is $\Ab=\{x,...,x,y\}$, we have $x>y$ with $p\xi_0+(q-1)x+y=\pi$, which gives
$$
q\le \frac{\pi}{x_0}-p\frac{\xi_0}{x_0}+1,
$$
and hence (\ref{pq1}) is proved.

Moreover in this case the set $\ola$ can be optimal only if it satisfies the optimality conditions which appears in Theorem \ref{lemmaAaAb}.
More precisely by Corollary 4.6 in \cite{HPR} (see (\ref{ineqHPR})) it must hold
$$
q-1\le -\frac{\mu_x}{\mu_y},
$$
where $\mu_x$ and $\mu_y$ are the eigenvalues of $D^2\jla{\la}(\oo)$ corresponding to the angles $x$ and $y$ respectively: $\mu_x= 2b\sin x(1-2\la b\cos x)$, $\mu_y=2b\sin y(1-2\la b\cos y)$. 
Indeed $\ola$ can be seen as an optimal polygon for the minimization  problem in the class of $(p+q)$-gones with $p$ fixed central angles equal to
$\xi_0$, and hence Corollary 4.6 in \cite{HPR} applies to the $q$ eigenvalues $\mu_x,...,\mu_x,\mu_y$.

We get the following necessary conditions for the optimality of the quasi-regular $N$-gone:
\begin{equation}\label{quasireg}
\begin{cases}
(q-1)x+y =\pi-p\xi_0,\\
\displaystyle{\cos x+\cos y =\frac 1{\la b}},\\
\sin x - (q-1)\sin y  \ge 0,\\
x-y >0.
\end{cases}
\end{equation}

Notice that this corresponds to find the intersections between the graph of the function
\begin{equation}\label{phila}
\phi_\la(x)=\arccos\Big(\frac 1{\la b}-\cos x\Big),
\end{equation}
and the straight line $y=\pi-(q-1)x-p\xi_0$, which belong to a certain subset of the first octant, as shown in Figure \ref{quasiregFigGen}.
\begin{center}
\begin{figure}[h]
\begin{tikzpicture}[x=4cm,y=1cm]
\shade[left color=gray,right color=white]  (0,0)--plot [domain=0:1.9] ({\x}, {(asin(0.5*sin(\x r))*pi/180)}) --(1.8,0) --cycle;
\draw (-0.2,0)--(1.8,0);
\draw (0,-0.5)--(0,3.5);
\draw (0,pi) -- (pi/2,0) node[near start,sloped, above]{\small$y=\pi-p\xi_0-(q-1)x$};
\draw[dashed] (0,0)--(pi/2,pi/2) node[near start, sloped, above]{\small$y=x$};
\draw[blue, thick] [domain=0.94:1.522] plot({\x}, {(acos(1.05-cos(\x r)))*pi/180});
\fill[blue] (1.51, 0.12) circle (2pt);
\fill[blue] (1.1, 0.94) circle (2pt);\draw[very thick, red] (1.1,0.94) node{\textsf{X}};
\draw[->, blue] (1.2,2) -- (1.2,0.9); \draw[blue] (1.4,1.8) node[above]{$\small\cos x+\cos y=\textstyle \frac 1{\la b}$};
\draw[blue!40!black!100!, thick] [domain=0:pi*0.6] plot({\x}, {(asin(0.5*sin(\x r))*pi/180)});
\draw[<-, blue!40!black!100!, thin] (1.8,0.6)--(1.8,0.9);\draw[blue!40!black!100!] (2,0.75) node[above]{\small$\sin x-(q-1)\sin y=0$};
\fill[black] (1.016,1.016) circle(1.2pt) node(B)[above]{\small$B$};
\draw[dashed, thin] (B)--(1.016,0) node[below]{\small$x_0$};
\fill[black] (1.39,0.515) circle(1.2pt) node(C)[above]{\small$C$};
\draw[dashed, very thin] (C)--(1.39,0) node[below]{\small$x_2$};
\fill[black] (1.52,0) circle(1.2pt) node[below]{\small$x_1$};
\draw[<-, thin] (1.525,-0.05)--(1.6,-0.2);\draw (1.64,-0.2) node{\small$A$};
\end{tikzpicture}
\caption{Conditions for the existence of a quasi-regular optimal polygon with $|\Aab|=p,|\Ab|=q$.}\label{quasiregFigGen}
\end{figure}
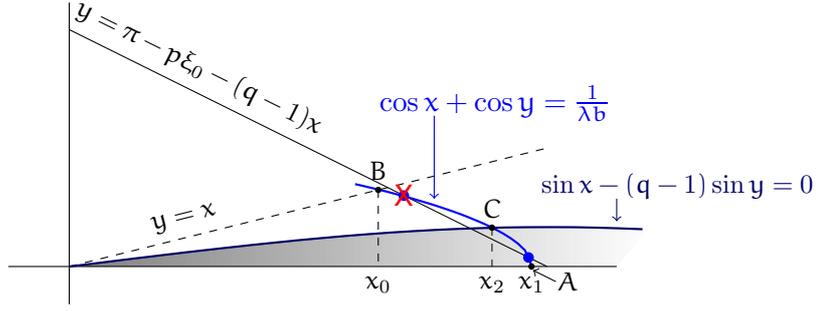
\end{center}
We denote by $A,B,C$ the points indicated in the figure: $A\equiv(x_1,0)$, $B\equiv(x_0,x_0)$, $C\equiv(x_2,y_2)$ such that
\begin{eqnarray}
&&\cos x_0 =  \frac 1{2\la b},\qquad \cos x_1 = \frac 1{\la b}-1, \label{x0x1} \\
&&\Big\{\cos x_2+\cos y_2= \frac 1{\la b},\qquad \sin x_2=(q-1)\sin y_2.\label{x2y2}
\end{eqnarray}

Hence we are interested in finding the zeros of the function
$$
\psi_\la(x)=\phi_\la(x)-\pi+p\xi_0+(q-1)x,
$$
which belong to the interval $[x_2,x_1)$.
Notice that the curve $\cos x+\cos y=1/\la b$ being concave (for $x>y>0$), so is the function $\psi_\la(x)$.
In particular $\psi_\la'(x)$ has a unique zero at the point $x_2$, since
$$
\phi'_\la(x_2)= -\frac{\sin x_2}{\Big(1-\big(\frac 1{\la b}-\cos x_2\big)^2 \Big)^{\frac 12}}=
-\frac{\sin x_2}{\big(1-\cos^2y_2\big)^{\frac 12}}=-\frac{\sin x_2}{\sin y_2}=-(q-1),
$$
and the function $\psi_\la$ is increasing for every $x\in(x_0,x_2)$ while it decreases for $x\in(x_2,x_1)$.
Hence there exists a zero for $\psi_\la$ in $[x_2,x_1)$ if and only if  $\psi_\la(x_2)\ge 0$ and $\psi_\la(x_1)\le 0$, that is
\begin{eqnarray}
&\phi_\la(x_2)-\pi+p\xi_0+(q-1)x_2 \ge& 0  \qquad\text{ and }\label{ai}\\
&\phi_\la(x_1)-\pi+p\xi_0+(q-1)x_1 \le& 0.\label{aii}
\end{eqnarray}
As  $\phi_\la(x_1)=0$, condition (\ref{aii}) yields
\begin{equation*}\label{upperN}
q\le \frac{\pi}{x_1}-p\frac{\xi_0}{x_1}+1,
\end{equation*}
which gives an upper bound to the number of possible chords (non-tangent to $D_a$) of an optimal polygon.

In order to find a lower bound for $q$ using (\ref{ai}), we need  to
estimate the value of $y_2=\phi_\la(x_2)$, which can explicitly be
found solving the system (\ref{x2y2}):
$$
y_2=\arccos\Big( \sqrt{1+\frac 1{b^2\la^2}\frac{(q-1)^2}{q^2(q-2)^2}} -\frac 1{b\la}\frac 1{q(q-2)} \Big).
$$
By algebraic computations one can prove that
$$
y_2\le\arccos\big(1-\frac 9{16 (q-1)^2} \big)\le \frac 32\sqrt{\frac 9{16 (q-1)^2}}=\frac 9{8(q-1)},
$$
and hence by (\ref{ai}) and the fact that $x_2<x_1$, we have
$$
\frac 9{8(q-1)}-\pi+p\xi_0+(q-1)x_1\ge y_2-\pi+p\xi_0+(q-1)x_2\ge 0,
$$
which implies
\begin{equation*}\label{lowerN}
q\ge \frac{\pi-p\xi_0+\sqrt{(\pi-p\xi_0)^2-{\textstyle\frac 92} x_1}}{2x_1}+1.
\end{equation*}
\end{proof}

\begin{corollary}\label{unicoN}
For $1/2b<\la\le 1/(a+b)$ there exists at most one $N\in\N$ such that an optimal polygon is a quasi-regular $N$-gone.
\end{corollary}
\begin{proof}
By Proposition \ref{labvuoto} and Proposition \ref{pq} we have
$$
\frac{\pi+\sqrt{\pi^2-{\textstyle\frac 92} x_1}}{2x_1}+1\le q\le \frac{\pi}{x_1}+1.
$$
Consider the difference between the upper and lower bounds:
\begin{eqnarray*}
\frac{\pi}{x_1}+1-\Big(\frac{\pi+\sqrt{\pi^2-{\textstyle\frac 92} x_1}}{2x_1}+1\Big)
&=&    \frac{\pi}{2x_1}-\frac{\pi}{2x_1}\sqrt{1-{\textstyle\frac 9{2\pi^2} x_1}}\\
&\le&  \frac{\pi}{2x_1}-\frac{\pi}{2x_1}\Big(1-\frac {27}{8\pi^2}x_1\Big)=\frac{27}{16\; \pi}<1,
\end{eqnarray*}
where the first inequality follows from the fact that $\textstyle\sqrt{1-u}\ge 1-\frac 34 u$ for $u\le \frac 89$.
%
%
Hence there exists at most one value of $N$ such that a quasi-regular optimal $N$-gone exists.
\end{proof}

As shown in the above proposition, quasi-regular optimal $N$-gones exist only for at most a specific value of $N$.
Hence in most cases the solution will be a regular polygon. 
In the following proposition we analyze more in details this situation. 
Notice that by Corollary \ref{unicoN} and Proposition \ref{regularinscrit} below we can characterize the number of sides of an optimal polygon, for $\la$ close to $1/2b$. 
In particular the number of sides tends to infinity as $\la$ tends to $1/2b$. 
This shows that we have some kind of continuity of the solutions of Problem (\ref{PB}) when $\la\to 1/2b$ and this is in contrast with the situation for $\la\to 2/a$. 
Indeed, as explained in Theorem \ref{teopxi0qx}, for $\la > 2/a$ the optimal solution $\ola$ has the minimum number of sides while for $\la>2/a$ it is the ball $D_a$.
\begin{proposition}\label{regularinscrit}
Let $1/2b<\la<1/b$ and consider the minimum Problem (\ref{PB}) in the class
$$
\CC\cap\{\oo \text{ regular polygon }\}.
$$
There exists a decreasing sequence $\{\hat\beta_n\}$ which tends to $1/4$, such that for $\la b/2\in[\hat\beta_N,\hat\beta_{N-1}]$ the optimum is either the polygon $P_N$ (if $P_N\in\CC$) or the polygon $P_m$ with $m$ the minimum such that $P_m\in\CC$ (if $P_N\not\in\CC$).
\end{proposition}
\begin{proof}
Let $1/2b<\la<1/b$ and let $P_N$ be a regular $N$-gone inscribed into $D_b$, we want to analyse the minimum of $\jla{\la}(P_N)$ with respect to $N$ and the value of $\la$, where
$$
\jla{\la}(P_N)=\pi b\Big( \la\;\frac b2\frac{\sin(2\pi/ N)}{2\pi/N}  -\frac{\sin(\pi/N)}{\pi/N} \Big).
$$
Let us denote $x=\pi/N$, and let $\beta=\la b/2$; with abuse of notation we will write $\jla{\la}(x)$ meaning $\jla{\la}(P_N)$.
Computing the derivatives of $\jla{\la}(x)$, we define $h(\beta,x)=x \,\jla{\la}'(x)$:
$$
h(\beta,x)= -\beta(\sin x\cos x-x\cos 2x)+\sin x-x\cos x.
$$
In order to study the minima of $\jla{\la}(x)$, we are interested in the zeros of $h$ for $1/4<\beta<1/2$, and $0<x\le \pi/3$.
We define the sequence $\beta_n$ such that $h({\beta_n},\pi/n)=0$, that is
\begin{equation}\label{betan}
\beta_n= \frac{\sin(\pi/n)-{\pi/n}\,\cos(\pi/n)}{\sin(\pi/n)\cos(\pi/n)-{\pi/n}\,\cos(2\pi/n)}.
\end{equation}
Notice that $\{\beta_n\}_{n\in\N}$ is a decreasing sequence which tends to $1/4$ as $n$ tends to infinity.

Consider $\beta_{n+1}<\beta<\beta_n$, then $h({\beta},\pi/n)$ is
positive while $h({\beta},\pi/(n+1))$ is negative hence $\jla{\la}$
has a minimum for $x\in[\pi/(n+1),\pi/n]$, which means that either
the optimal number of sides is $n$ or it is $(n+1)$. In particular
there exists $\hat\beta_n\in[\beta_{n+1},\beta_n)$ such that
$\jla{2\beta/b}(\pi/n)$ is minimum for $\beta\in
[\hat\beta_n,\hat\beta_{n-1}]$ where
$$
\hat\beta_n=\Big(\frac{\sin(\pi/n)}{\pi/n} -\frac{\sin(\pi/(n+1))}{\pi/(n+1)}\Big)\Big/
\Big( \frac{\sin(2\pi/n)}{2\pi/n}-\frac{\sin(2\pi/(n+1))}{2\pi/(n+1)}\Big),
$$
and $\jla{\hat\la_n}(P_n)=\jla{\hat\la_{n+1}}(P_{n+1})$ for $\hat\la_n=2\hat\beta_n/b$.

Hence let $\bar n$ be the minimum number of sides such that $P_{\bar n}$ belongs to the class $\CC$ and consider $1/2b<\la<1/b$.
Let $n\in\N$ be such that $2\la/b\in [\hat\beta_n,\hat\beta_{n-1}]$.
If $n\ge\bar n$, then $P_n$ minimizes $\jla{\la}$ among all regular polygons, if $n<\bar{n}$ then the optimal is $P_{\bar n}$.
\end{proof}
Notice that this result implies that in the case $b\ge 2a$, and $\hat\beta_3\le 2b\la\le 1/2$ the optimal regular polygon is the equilateral triangle.

More generally in the case $b\ge 2a$ and $1/(a+b)\le\la\le 1/b$, we are going to show that only triangle can be optimal sets.
\begin{proposition}\label{rem3.8}
Let $b> 2a$ and $\frac 1{a+b}\le\la\le \frac 1b$; then $\ola$ is a triangle.
\end{proposition}
\begin{proof}
As $\la\le 1/b\le 2/(a+b)$, the class $\Aa$ is empty by Lemma \ref{Lavuoto}. 
We split  the proof in two parts, considering the two cases $\Aab=\emptyset$ and $\Aab\neq\emptyset$.

Assume $\ola$ have no tangent sides to $D_a$ (that is $\Aab=\emptyset$) and that $\ola$ is a quasi-regular polygon; hence condition (\ref{quasireg}) must hold true.
Consider the curve $\cos x+ \cos y = 1/b\la$; as $1/(a+b)\le\la\le1/b$ and $b> 2a$, it holds
$$
1\le\frac 1{\la b}\le \frac 32.
$$
We compare the graphs of the functions $y=\arccos({\textstyle \frac 1{b\la}}-\cos x)$ in the extreme cases $1/(b\la)=1$ and $1/(b\la)=3/2$.

Applying Proposition \ref{pq} we get either $N=3$ or $N=4$, that is: between quasi-regular polygons, only triangles and quadrilaterals can be optimal sets. 
Indeed for each $N\ge 5$ there is no intersection between the curve $\cos x+\cos y=1/2\la b$ and the line $y=\pi-(N-1)x$ as shown in Figure \ref{quasiregFig} (a).
In particular quadrilaterals are not optimal as the (non null) values of $x$ such that there exists a solution to
$$
\Big\{\cos x+ \cos y =\frac 1{b\la},\qquad y=\pi-(N-1)x,
$$
for $N=4$, does not satisfy $\sin x\ge(N-1)\sin y$, as shown in Figure \ref{quasiregFig} (b).
Hence the only possible quasi-regular optimal polygons with $\Aab=\emptyset$ are triangles of the form $\Ab=\{x,x,y\}$.
\begin{center}
\begin{figure}[h]
\begin{tikzpicture}[x=3cm,y=1.3cm]
\draw (-0.2,0)--(1.8,0);
\draw (0,-0.5)--(0,3.5);
\filldraw[draw=black,fill=gray!20]  plot[domain=1.05:1.5715] ({\x}, {(acos(1-cos(\x r)))*pi/180}) --
plot[domain=1.0475:0.72] ({\x},{(acos(1.5-cos(\x r)))*pi/180})--cycle;
\draw (0,pi) -- (pi/2,0) node[near start,sloped, above]{\tiny$N=3$};
\draw (0,pi) --(pi/3,0) node[midway,sloped, above]{\tiny$N=4$};
\draw (0,pi)--(pi/4,0) node[midway,sloped, below]{\tiny$N=5$};
\draw[dashed] (0,0)--(pi/2,pi/2) node[near start, sloped, above]{\tiny$y=x$};
\draw[thick] [domain=1:1.5715] plot({\x}, {(acos(1-cos(\x r)))*pi/180});
\draw[->] (1.3,1.7)node[above]{$\textstyle{\frac 1{b\la}=1}$} --(1.3,0.8);
\draw[thick] [domain=0.68:1.0475] plot({\x},{(acos(1.5-cos(\x r)))*pi/180});
\draw[->] (1,2.5)node[above]{$\textstyle{\frac 1{b\la}=\frac 32}$}--(1,0.5);
\draw (1,-1) node{(a)};
\begin{scope}[xshift=220]
\draw (-0.2,0)--(1.8,0);
\draw (0,-0.5)--(0,3.5);
\draw [domain=0:pi/2] plot (\x,{pi-\x*2});
\draw [domain=0:pi/3] plot (\x,{pi-\x*3});
\draw[dashed] (0,0)--(pi/2,pi/2);
\draw[blue, thick] [domain=0.95:1.5715] plot({\x}, {(acos(1-cos(\x r)))*pi/180});
\fill[blue] (1.047197551,1.047197551) circle (2pt);
\draw[blue, thick] [domain=0.9:1.4456] plot({\x}, {(acos(1.125-cos(\x r)))*pi/180});
\fill[blue] (1.318116072, 0.505360511) circle (2pt);
\draw[thick] [domain=0.68:1.0475] plot({\x},{(acos(1.5-cos(\x r)))*pi/180});
\fill (0.87,0.53) circle (2pt);
\draw[thick] [domain=0.72:1.1452] plot({\x},{(acos(1.414-cos(\x r)))*pi/180});
\fill (0.785398164,0.785398164) circle (2pt);
\draw[blue, thin] [domain=0:pi/2] plot({\x}, {(asin(0.5*sin(\x r))*pi/180)});
\draw[<-, blue, thin] (1.5,0.6)--(1.5,0.85)node[above]{\tiny$\sin x-2\sin y=0$};
\draw[thin] [domain=0:pi/2.7] plot({\x}, {(asin(0.333*sin(\x r))*pi/180)});
\draw[->, thin] (0.5,-0.2) node[right]{\tiny$\sin x-3\sin y=0$} --(0.5,0.1);
\draw (1,-1) node{(b)};
\end{scope}
\end{tikzpicture}
\caption{Conditions for the existence of a quasi-regular optimal polygon with $\Aab=\emptyset$.
Case $b\ge 2a$,  $\textstyle \frac 1{a+b}\le\la\le\frac 1b$.}\label{quasiregFig}
\end{figure}
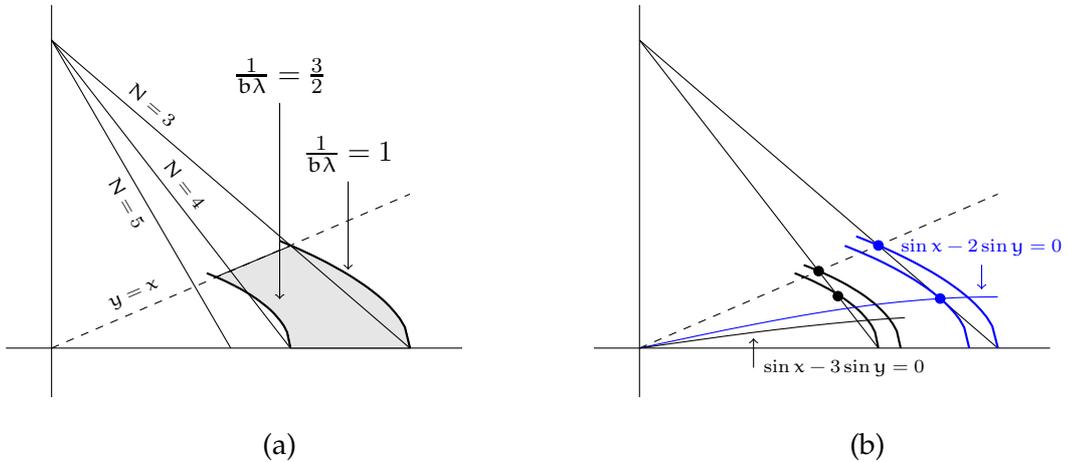
\end{center}
Consider now the case of a regular $N$-gone $P_N$; it holds
$$
\jla{\la}(P_N)= 2b\,N\sin\textstyle{\frac{\pi}N}\Big( \la\frac b2\cos\textstyle{\frac{\pi}N}-1 \Big).
$$
Notice that, as $1/(a+b)\le \la\le 1/b$ with $b> 2a$, we have $\la b/2\in(1/3,1/2)$ and hence Proposition \ref{regularinscrit} guarantees $N=3$.


Hence if $\Aab$ is empty necessarily $\ola$ is a triangle; either equilateral or isosceles.

Suppose now $\Aab$ to be not empty; as $b> 2a$ it holds $|\Aab|=p\le 2$ and  Proposition \ref{pq} guarantees that $\ola$ is either a triangle or a quadrilateral. 
We are going to show that in fact this latter cannot arrive. 
Assume $\Ab\supseteq\{x,y\}$ with $x>y$ and let $q=|\Ab|\ge 2$. 
By the first order optimality conditions we have
\begin{equation}\label{1pxi0qx}
\begin{cases}
y=\pi-p\xi_0-(q-1)x,\\
x>y\\
\cos x+\cos y=\frac 1{\la b},
\end{cases}
\end{equation}
where it holds $1\le{\textstyle\frac 1{\la b}\le \frac 32}$ and $\textstyle\frac {\pi}3<\xi_0<\frac{\pi}2$.

Notice that the constant term and the director coefficient of the line in (\ref{1pxi0qx}) decreases with respect to $p$ and $q$, respectively.
Hence if (\ref{1pxi0qx}) admits no solution for some $\bar p, \bar q$, then the same will arrive for every $p\ge \bar p, q\ge \bar q$.
\begin{center}
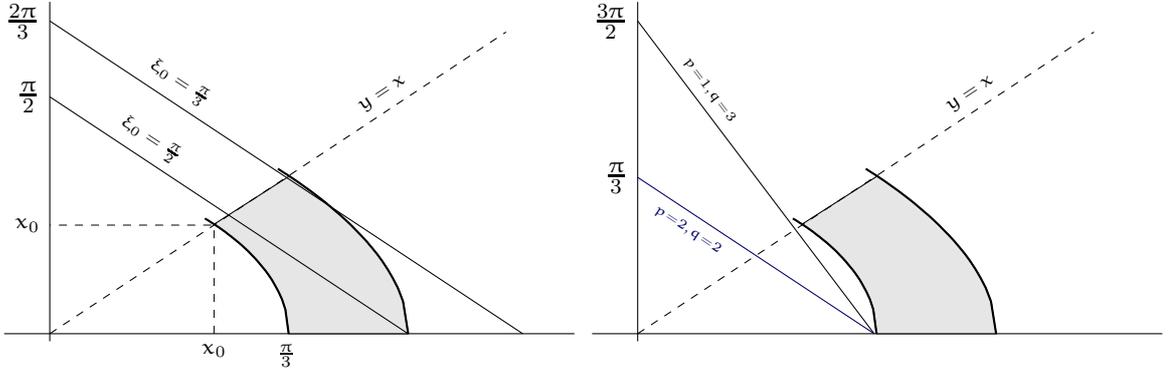
\begin{figure}[h]
\begin{tikzpicture}[x=3cm,y=2cm]
\draw (-0.2,0)--(2.3,0);
\draw (0,-0.05)--(0,2.2);
\filldraw[draw=black,fill=gray!20]  plot[domain=1.05:1.5715] ({\x}, {(acos(1-cos(\x r)))*pi/180}) --
plot[domain=1.0475:0.72] ({\x},{(acos(1.5-cos(\x r)))*pi/180})--cycle;
\draw (0,pi*0.66) node[left]{$\frac{2 \pi}3$} -- (0.66*pi,0) node[near start,sloped, above]{\tiny$\textstyle{\xi_0=\frac{\pi}3}$};
\draw (0,pi*0.5) node[left]{$\frac{\pi}2$} -- (0.5*pi,0) node[near start,sloped, above]{\tiny$\textstyle{ \xi_0=\frac{\pi}2}$};
\draw (pi*0.33,0) node[below]{$\scriptstyle \frac {\pi}3$};
\draw[very thin, dashed] (0.72,0) node[below]{$\scriptstyle x_0$} --(0.72,0.72)--(0,0.72) node[left]{$\scriptstyle x_0$};
%
%
\draw[dashed] (0,0)--(2,2) node[near end, sloped, above]{\tiny$y=x$};
\draw[thick] [domain=1:1.5715] plot({\x}, {(acos(1-cos(\x r)))*pi/180});
\draw[thick] [domain=0.68:1.0475] plot({\x},{(acos(1.5-cos(\x r)))*pi/180});
\draw (1,-1) node{(a) Case $p=1, q=2$.};
\begin{scope}[xshift=220]
\draw (-0.2,0)--(2.3,0);
\draw (0,-0.05)--(0,2.2);
\filldraw[draw=black,fill=gray!20]  plot[domain=1.05:1.5715] ({\x}, {(acos(1-cos(\x r)))*pi/180}) --
plot[domain=1.0475:0.72] ({\x},{(acos(1.5-cos(\x r)))*pi/180})--cycle;
\draw (0,pi*0.66) node[left]{$\frac{3\pi}2$} --(pi*0.33,0) node[near start,sloped, above]{\tiny$\scriptstyle{p=1, q=3}$};
\draw[blue!40!black!100!, thin] (0,pi*0.33)-- (0.33*pi,0) node[near start,sloped, below]{\tiny$\scriptstyle{p=2, q=2}$};
\draw (0,pi*0.33) node[left]{$\frac{\pi}3$};
\draw[dashed] (0,0)--(2,2) node[near end, sloped, above]{\tiny$y=x$};
\draw[thick] [domain=1:1.5715] plot({\x}, {(acos(1-cos(\x r)))*pi/180});
\draw[thick] [domain=0.68:1.0475] plot({\x},{(acos(1.5-cos(\x r)))*pi/180});
\draw (1,-1) node{(b) Case $p=1, q=3$ and $p=2, q=2$.};
\end{scope}
\end{tikzpicture}
\caption{Conditions for the existence of an optimal polygon with
$\Aab\neq\emptyset$, $\Ab\supseteq\{x,y\}$. Case $b> 2a$,  $\textstyle \frac 1{a+b}\le\la\le\frac 1b$.}\label{quasiregFigQ}
\end{figure}
\end{center}
Consider the case $p=q=2$, shown in Figure \ref{quasiregFigQ} (b).
Notice that, the line $y=\pi-2\xi_0-x$ never intersects the curve $\cos x+\cos y=\frac 32$ for $x>y>0$ (and hence it never intersects $\cos x+\cos y=1/\la b$ neither). 
Indeed, thanks to the concavity of the function $\phi_{\scriptstyle\frac 2{3b}}(x)=\arccos(3/2-\cos x)$, the curve $\cos x+\cos y= 3/2$ for $x>y>0$ stays above the line through the points $(\pi/3,0)$, $(x_0,x_0)$, where $x_0=\arccos 3/4$, and this latter stays above the line $y=\pi-2\xi_0-x$ for every $y>0$. Hence there is no solution to (\ref{1pxi0qx}) for $p=q=2$. 
The same arrives for $p=1,q=3$ as shown again in Figure \ref{quasiregFigQ} (b). 
This implies that the only possible case is $p=1, q=2$, which corresponds to an isosceles triangle whose central angles are $\{\xi_0,z,z\}$, represented in Figure \ref{quasiregFigQ} (a).

Assume now that $\Ab$ only contains copies of the same angle $x$, with $|\Ab|=q\ge 2$ and $p\xi_0+q x=\pi$.
By the second order optimality conditions (see Remark \ref{rmkoptimality}), and the fact that $\la\ge 1/(a+b)\ge 2/3b$, we have
$$
\cos x\le \frac 1{2\la b}\le \frac{a+b}{2b}\le \frac 34,
$$
that is $x\ge u_0\ge x_0$, where $u_0$ is such that $\cos u_0=(a+b)/2b$.
Hence we have
\begin{equation}\label{2pxi0qx}
3-p\le q\le \frac {\pi}{u_0}-p\frac{\xi_0}{u_0}\le \frac {\pi (3-p)}{3\,u_0},
\end{equation}
where $p=1,2$ by construction, as $b> 2a$ and $\xi_0\ge \pi/3$.
Let us analyse these cases separately; notice that $u_0\ge x_0 =\arccos(3/4)\ge 0.72$.

For $p=1$ we obtain $2\le q\le 2.9$, which implies that the only possible polygon with $\Aab=\{\xi_0\}$ is the triangle with $\Ab=\{x,x\}$. 
For $p=2$ condition (\ref{2pxi0qx}) reads as $1\le q\le 1.44$ which gives  $q=1$ and hence again the only possibility is a triangle, which can be identified by its central angles as $\{\xi_0,\xi_0,z\}$.

\bigskip

Hence the optimal polygons are triangles, in particular they are of the form:
\begin{eqnarray*}
T=\{\pi/3,\pi/3,\pi/3\},\quad T'=\{x,x,y\},\qquad  T''=\{\xi_0,z,z\},\quad T'''=\{\xi_0,u,v\}\quad T'^{v}=\{\xi_0,\xi_0,w\},
\end{eqnarray*}
where the polygons are indicated using their central angles and
$z=\frac{\pi-\xi_0}2$, $w=\pi-2\xi_0$ are fixed. The values of $x,y$
and $u,v$ are given accordingly to Theorem \ref{lemmaAaAb}. It is
possible to simply compare these five kind of triangles by splitting them in
two (non disjoint!) classes:
\begin{itemize}
  \item the class of triangles with at least one central angle
$\xi_0$;
  \item the class of isosceles triangle.
\end{itemize}

Let us consider first the class of triangles with at least one central angle $\xi_0$. 
All of them can be represented as triangles whose central angles are $\{\xi_0,u,\pi-\xi_0-u\}$ with $u\in(\textstyle \frac{\pi}{2}-\frac{\xi_0}{2}, \xi_0)$. 
Notice that the limit cases $u=\textstyle \frac{\pi}{2}-\frac{\xi_0}{2}$ and $u=\xi_0$ correspond to the triangles $T''$ and $T'^v$ respectively.
Writing down the functional $\jla{\la}$ as a function of $u$, we get
three different optimal triangles depending on the value of $\la$:
\begin{eqnarray*}
T''=\{\textstyle\xi_0,\frac{\pi}2-\frac{\xi_0}2,\frac{\pi}2-\frac{\xi_0}2\}&\qquad\text{ for }& \textstyle \frac 1{a+b}\le  \la \le \frac{2b^2}{(b-a)(b+2a)},\\
T'''=\{\xi_0,\bar u,\pi-\xi_0-\bar u\}&\qquad{ for }& \textstyle \frac 1{\sqrt{2b} \sqrt{b-a}}\le \la  \le\frac{2b^2}{(b-a)(b+2a)},\\
T'^v=\{\xi_0,\xi_0,\pi-2\xi_0\}&\qquad\text{ for }& \textstyle \frac{2b^2}{(b-a)(b+2a)}\le  \la \le \frac 1b.
\end{eqnarray*}
where $\bar u$ is such that $\textstyle \sin(\bar u+\frac{\xi_0}2)(2\la b\sin{\frac{\xi_0}2)}=1$.
Hence there exists only one possible optimal triangle of the type $T'''$ corresponding to $u=\bar u$.
%
%

On the other hand, in the class of isosceles triangles determined by their central angles $\{x,x,y\}$, with $x\in[\frac {\pi}3,\xi_0]$, we have seen that there exists at most one triangle of type $T'$ which can be optimal. 
More precisely the functional $\jla{\la}$ is an increasing function of $x$ if $\la\le \frac 8{9b}$ or if $2a\le b\le 4a$ for every $\la$ and hence in these cases the only possible optimal isosceles triangles is $T'^v$. 
In the case $b> 4a$ with $\frac 8{9b}<\la\le\frac 1b$ there could exist a unique optimal triangle $T'$, corresponding to the unique possible point $\xx$ of
local minimum for $\jla{\la}$:
$$
\cos \xx= \frac{1+\sqrt{9-8\la b}}4.
$$

Hence for each $\frac 1{a+b}\le\la\le\frac 1b$ the solution to Problem  (\ref{PB}) is a triangle and the comparison between the two above classes yields the precise optimal one. 
Let us remark that, using a straightforward but tedious calculation, it is possible to prove that the optimal triangle is always one of the following: $T$, $T'$ with $x=\xx$ or $T'^v$.
\end{proof}


\section{An example}\label{examples}
Let us consider in detail an example to explain how the previous results allow us to easily get any solution of the problem for any value of the parameter $\la$. 
We choose here to fix $a=1,b=3$. 
Then $\xi_0=\arccos(a/b)\simeq 1.2310$. 

The cases $\la> \frac{2}{a}=2$ and $\la\leq \frac{1}{2b}=\frac{1}{6}$  are covered by Theorem \ref{teoDaDb} and the solutions are respectively $D_a$ and $D_b$.

For $\la=2$, as explained in Remark \ref{la=2/a}, any polygon circumscribed to $D_a$ and any combination of sides tangent to $D_a$ and arcs of the circle $D_a$ solves the problem.

Let us consider the case $1/2b<\la<2/a$.
First we want to apply Theorem \ref{teopxi0qx}. 
Since $\xi_0\simeq 1.2310$, we have $p=2$ and $x=\pi-2\xi_0\simeq 0.6797$.
The critical value of $\la$ which is equal to $2/(b\cos x +a)$ equals
$$
\tilde{\la}=\frac{2}{1-3\cos 2\xi_0}\,=\frac{2b}{b^2+ab-2a^2}\,=0.6\,.
$$
Therefore, for $\la\geq 0.6$, the optimal solution is the isosceles triangle  circumscribed to $D_a$ while for $1/3<\la\leq 0.6$ the optimal solution is the isosceles triangle inscribed into $D_b$, see Table \ref{example1,3}.

Now for $\la$ between $0.25=1/(a+b)$ and $1/b$, we use the analysis done in Proposition \ref{rem3.8} and the comparison between all triangles. This shows that the isosceles triangle inscribed into $D_b$ (and defined by its three angles $\xi_0,\xi_0,\pi-2\xi_0$ remains the optimal domain for $\la\in (0.308, 1/3)$ while the equilateral triangle becomes the optimal domain for $\la \in (0.25,0.3080)$.

For $\la < 1/(a+b)=0.25$, according to Proposition \ref{labvuoto},
we know that the optimal domain is inscribed in $D_b$ (and does not
touch $D_a$). Moreover, by Proposition \ref{regularinscrit}, we are
able to compare all regular polygons. More precisely, the following
table shows the values of $\la$ for which we switch from the regular
$N$-gone to the regular $(N+1)$-gone (e.g. we switch from the
equilateral triangle to the square for $\la\leq 0.2191$).

\begin{center}
    \begin{tabular}{c|c|c|c|c|c|c|c|c}
   $\la$ & 0.2191 & 0.1951 & 0.1847 & 0.1792 & 0.1759 & 0.1738 & 0.1723 & 0.1713\\ \hline
       N &  3     & 4      & 5      & 6      & 7      & 8      & 9      & 10
   \end{tabular}
\end{center}

Now we have seen in Theorem \ref{lemmaAaAb} that the only other possible optimal domain is a {quasi-regular} polygon with $N-1$ angles $x$ and one angle $y=\pi-(N-1)x$. 
Moreover, Proposition \ref{unicoN} shows that there exists at most one possible value of $N$ for such a quasi-regular polygon (and we have explicit bounds for this $N$), therefore the numerical study is easy. 
In our case, it turns out that we are able to find such quasi-regular polygons only twice (for two small intervals):
\begin{itemize}
  \item If $\la \in (0.21874;0.22222)$ the optimal domain is a
  quasi-regular quadrilateral.
  \item If $\la \in (0.19506;0.19525)$ the optimal domain is a
  quasi-regular pentagon (with a very small angle $y$, thus it is not
  easy to recognize a pentagon in the corresponding Figure of
  Table \ref{example1,3}).
\end{itemize}
For the other values of $\la$, the optimal domain is the regular
$N$-gone and we just have to follow the Table in the Appendix (Section \ref{tableBn}). Thus, we
have represented the solutions in Table \ref{example1,3} only up to
the regular hexagon. Let us remark that, in this table, the values
of the angles $x$ and $y$ for the quasi-regular polygons are given
as an example for one choice of $\la$.

{\scriptsize
\begin{table}[h]
\begin{tabular}{|c|l|c|c|c|c|c|}\hline
Interval for & Optimal & \multicolumn{3}{c|}{Class of Angles} & & \\
  $\lambda$         &  Solution     & $\Aab$        & $\Aa$       & $\Ab$       & Figure       & Area \\ \hline
  $(2;+\infty)$     &  disk $D_a$   & $\emptyset$  & $\emptyset$  & $\emptyset$ & \Da          & $\pi$ \\ \hline
  $(0.6;2)$         & \isotri    & $1.2310\times 2$ & 0.6797   & $\emptyset$ & \isoscUno    & $6.4650$ \\ \hline
  $(0.3080;0.6)$    & \isotri    & $1.2310\times 2$ & $\emptyset$ & 0.6797   &\isoscDue    & $10.0566$ \\ \hline
  $(0.2222;0.3080)$ & \equitri   & $\emptyset$  & $\emptyset$ & $\frac{\pi}3\times 3$&\equi & $11.6913$ \\ \hline
  $(0.2187;0.2222)$ & \qrquadr & $\emptyset$  & $\emptyset$ & \trexy & \quadriirreg & 13.0245 \\ \hline
  $(0.19525;0.2187)$& square        & $\emptyset$  & $\emptyset$ & $\frac{\pi}4\times 4$&\squareFigure & 18 \\  \hline
  $(0.19506;0.19525)$& \qrpent & $\emptyset$  & $\emptyset$ & \begin{tabular}{r}$4\times x=0.7829$\\$y=0.0098$\end{tabular} & \pentairregUno &  18.0879 \\ \hline
  $(0.1847;0.19506)$& \regpent    & $\emptyset$  & $\emptyset$ & $\frac{\pi}5\times 5$&\pentag & 21.3988 \\ \hline
  $(0.1792; 0.1847)$& \regex    & $\emptyset$  & $\emptyset$ & $\frac{\pi}6\times 6$&\esag  &   23.3827    \\ \hline
  (2$\hat{\beta}_N/3;2\hat{\beta}_{N-1}/3$)& \regNgon & $\emptyset$ & $\emptyset$ & $\frac{\pi}N\times N$ & \vdots & \vdots\\ \hline
  $(0;1/6)$         &  disk $D_b$   & $\emptyset$  & $\emptyset$ & $\emptyset$  &\Db            & $9\pi$  \\ \hline
\end{tabular}
\caption{Optimal sets for $a=1,b=3$, $0\le\la\le+\infty$.}\label{example1,3}
\end{table}
}
\section{Some related problems}

In this section we begin by investigating the same problem when we
remove one of the unilateral constraint $D_a\subset\Omega$ or
$\Omega\subset D_b$. We show that the previous study allows to
handle also these cases. Then, choosing particular values for the
parameter $\la$, we are able to recover a classical inequality due
to Bonnesen and Fenchel involving the area, the perimeter and the
inradius. Then we recover another one due to J. Favard which involves the area, the perimeter and the circumradius. 
We are also able to find a refinement of such inequality for large perimeter. 
We close this section with a discussion about the problem of maximizing perimeter with a volume constraint in the class $\C$.

\subsection{Variation of constraints}
It is interesting to investigate Problem (\ref{PB}) with different constraints. 
In particular it is often useful to consider convex sets which either contain a common fixed ball or which are contained
in it. This corresponds to consider the class of convex sets with
not too small inradius, or on the opposite side, the class of not
too large convex sets.

\subsubsection{Analysis of convex sets with not too small inradius}
For a fixed positive real number $a$ we define the class $\inr{a}$ as
the class of convex sets which contain the ball $D_a$ and we consider the problem
\begin{equation}\label{PBa}
\min_{\oo\in\inr{a}} \jla{\la}(\oo),
\end{equation}
where $\jla{\la}$ is defined as in (\ref{Jla1}).

Notice that not for every values of $\la$ a solution exists. 
Indeed for small values of $\la$ the perimeter has  in fact the heaviest weight, and it is not bounded. More precisely, solutions to (\ref{PBa}) can be seen as limit of solutions to Problem (\ref{PB}) in the class $\C$ for $b$ which tends to infinity. 
Hence for $0\le\la<\frac 2a$ a possible solution should be the limit of the triangle $T''$ in Figures \ref{triangles} (c). 
However $\lim_{b\to\infty}\jla{\la}(T'')=-\infty$ and hence a minimum does not exists.

More generally, as for values of $\la\ge \frac 2a$  solutions to (\ref{PB}) do not depend on the exterior ball $D_b$, they solve Problem (\ref{PBa}) as well. 
Indeed let $\ola$ be a solution to (\ref{PBa}); then either $\ola$ is contained in a ball $D_b$ or it is a limit of a sequence $\{\oo_n\}$ with $\oo_n\subseteq D_{b_n}$ for some $b_n$, since otherwise the functional could not be defined.
Hence we can apply the analysis of Problem (\ref{PB}) and we get the following.
\begin{proposition}
For $\la<\frac 2a$ there is no solution to Problem (\ref{PBa})  while for $\la\ge 2/a$ solutions exist and they coincide with the corresponding solutions to Problem (\ref{PB}). 
More precisely for $\la=2/a$ there exist an infinite number of solutions, which are circumscribed figures composed by arcs of $D_a$ and tangent segment, while for $\la> \frac 2a$ the ball $D_a$ is the unique solution.
\end{proposition}

\subsubsection{Analysis of not too large convex sets}\label{secPBb}
For $b>0$ we define the class $\outr{b}$ as the class of convex sets contained in the ball $D_b$ and we consider the problem
\begin{equation}\label{PBb}
\min_{\oo\in\outr{b}} \jla{\la}(\oo),
\end{equation}
where $\jla{\la}$ is defined as in (\ref{Jla1}).

Since for every fixed $b>0$ the class $\outr{b}$ is compact  for the
Hausdorff distance, the existence of a solution to Problem
(\ref{PBb}) is guaranteed for every $\la\ge 0$. We would like to
solve Problem (\ref{PBb}) passing to the limit $a\to 0$ in Problem
(\ref{PB}), but this cannot be done directly since we cannot assume
that an optimal set $\ola$ to (\ref{PBb}) contains the ball $D_a$,
even for very small $a>0$. However we can circumvent this difficulty
by considering a ``translated'' problem.

Let $\oo\in\outr{b}$ be given. If the origin is in the exterior of
$\oo$, it means that $\oo$ lies in an open half-disc and we can
translate it (without changing the value of the functional) to
assume either that the origin is in the interior of $\oo$ or that it
is on its boundary . If the origin is in the interior of $\oo$ there
exists $\ep>0$ such that $\oo\in\mathscr{C}_{\ep,b}$ which entails
\begin{equation}\label{Cb1}
\jla{\la}(\oo)\ge\jla{\la}(\oo_\la^\ep),
\end{equation}
where $\oo_\la^\ep$ is an optimal set for the Problem (\ref{PB}) in
the class $\mathscr{C}_{\ep,b}$. We can now use the analysis done
for Problem (\ref{PB}). Hence for $\la\le 1/2b$ the set
$\oo_\la^{\ep}$ is the ball  $D_b$, while for $\frac 1{2b}<\la\le
\frac 1{b+\ep}$ the set $\oo_\la^\ep$ is strictly inscribed into
$D_b$ and it is either regular or quasi-regular. For $\frac
1b\le\la\le \frac{2b}{(b-\ep)(b+2\ep)}$ we have $\oo_\la^\ep=T'_b$
the set in Figure \ref{triangles} (b) whose circumradius is $b$ and
inradius is $\ep$, while for $\frac{2b}{(b-\ep)(b+2\ep)}\le \la \le
\frac 2{\ep}$ the set $\oo_\la^\ep$ is the triangle $T''_b$ in
Figure \ref{triangles} (c), with circumradius $b$ and inradius
$\ep$. Passing to the limit for $\ep$ which tends to zero we get
inequality (\ref{Cb1}) with $\ola$ equal to the optimal set of
Problem (\ref{PB}) for $0\le\la\le1/b$, while for $\la\ge 1/b$ we
obtain as optimal set a double diameter.

If $\oo$ contains the origin on its boundary then we consider  a
translation of the origin such that $O_\ep=O-\ep$,
$\oo_\ep=\oo-\ep$. Hence $\oo_\ep\in\mathscr{C}_{\ep,b_\ep}$ for
sufficiently small $\ep$ and $b_\ep=b+\ep$. As $|\oo_\ep|=|\oo|,
P(\oo_\ep)=P(\oo_\ep)$, inequality (\ref{Cb1}) still holds true,
with $\oo_\la^\ep$ an optimal set for the Problem (\ref{PB}) in the
class $\mathscr{C}_{\ep,b_\ep}$. The same argument as before
(passing to the limit when $\ep\to 0$) leads to the following
result.

\begin{proposition}\label{propJb}
For every $\la\ge 0$ there exists a solution $\ola$ to the problem (\ref{PBb}).
In particular $\ola$ coincides with the optimal set in Problem (\ref{PB}) for $\la < 1/b$, while $\ola$ is a double diameter for $\la\ge 1/b$.
\end{proposition}

\subsection{Inequalities for convex sets}
In the study of the theory of convex sets, geometric inequalities play a crucial rule as they allow to connect important geometric quantities (as the area and the perimeter) and to have an estimate of them.
We refer to \cite{SA} for a summary of the most famous.

\subsubsection{Area, perimeter and inradius}
A well known inequality involving the area $|\oo|$, the perimeter $P(\oo)$ and the inradius $r(\oo)$ of a convex set $\oo$ is due to Bonnesen and Fenchel (see \cite{BF}).
They proved that for every planar convex set $\oo$,
\begin{equation}\label{BonFen}
P(\oo)\le 2\frac{|\oo|}{r(\oo)}.
\end{equation}

Notice that Theorem \ref{teoDaDb} offers a new proof of this result. 
Indeed: let $\oo$ be a planar convex set, up to translation of the origin we can assume $D_r\subset\oo$, where $r=r(\oo)$; moreover there exists $R>r$ such that $\oo\subset D_R$ and hence $\oo\in\mathscr{C}_{r,R}$.
Then Remark \ref{la=2/a} entails
$$
\frac 2r |\oo|-P(\oo)\ge \frac 2r |D_r|-P(D_r)=0,
$$
which corresponds to Bonnesen-Fenchel inequality (\ref{BonFen}) and in particular equality holds in (\ref{BonFen}) for every polygon circumscribed to $D_a$ as well as for every convex set $\oo$ whose boundary is composed by arcs of $D_r$ and tangent sides to it.

\subsubsection{Area, perimeter and circumradius}\label{AreaPerimeterCircumradius}
Another interesting inequality regards the area, the perimeter and the circumradius $R(\oo)$.
In \cite{F} it is proved that for every planar convex set $\oo$ it holds
\begin{equation}\label{circumradius}
|\oo|\ge R(\oo)(P(\oo)-4R(\oo)),
\end{equation}
with equality for linear segments.

Using Theorem \ref{teopxi0qx} for $\la=1/b$, we can recover this result.
Indeed, let $\oo$ be a planar convex set and let $R=R(\oo)$ be its circumradius; up to translation of the origin we can assume $\oo\subseteq D_R$.
If $\oo$ contains the origin in its interior, then there exists $\ep>0$ such that $D_\ep\subset\oo$ and hence $\oo\in\mathscr{C}_{\ep,R}$, which implies
\begin{equation}\label{circumradius2}
\frac 1{R}|\oo|-P(\oo)\ge \frac 1R |T_\ep'|-P(T'_\ep)=4{\sqrt{R^2-\ep^2}}\Big( \frac {\ep^3}{R^3}-\frac {\ep}R -1 \Big),
\end{equation}
where $T'_\ep$ is the triangle in Figure \ref{triangles} (b), whose inradius is $\ep$.
Passing to the limit for $\ep$ which tends to zero, we obtain
$$
\frac 1{R}|\oo|-P(\oo)\ge -4R,
$$
with equality for segments, which are in fact obtained as limits of triangles $T'_\ep$.
If the origin is on the boundary of $\oo$ then using the same argument than in Section \ref{secPBb} we have $\oo_\ep=\oo-\ep\in\mathscr{C}_{\ep,R+\ep}$.
Applying Theorem \ref{teopxi0qx} for $\la=1/(R+\ep)$, we get inequality (\ref{circumradius2}) for $R_{\ep}=R+\ep$,
 \begin{equation*}
\frac 1{R_\ep}|\oo|-P(\oo)\ge 4{\sqrt{R_\ep^2-\ep^2}}\Big( \frac {\ep^3}{R_\ep^3}-\frac {\ep}R_\ep -1 \Big),
\end{equation*}
and passing to the limit for $\ep$ which tends to zero, we get (\ref{circumradius}), with equality for diameters of the ball $D_R$.

\bigskip
Actually, we can get another similar inequality which improves the
previous one for "large" perimeter. Indeed if we choose now
$\la=1/2b$ in Proposition \ref{propJb}, the optimal domain is the
ball $D_b$. Thus, for any domain included in the ball $D_b$, the
following inequality holds
$$\frac{1}{2b}|\oo|-P(\oo)\geq \frac{\pi b^2}{2b}-2\pi b=-\frac{3\pi
b}{2}\,.$$ In particular, replacing $b$ by the circumradius yields
the following proposition:

\begin{proposition}\label{propcircum}
For a convex set $\oo$ the following
inequality holds
\begin{equation}\label{ineqcirc}
|\oo|\ge R(\oo)(2 P(\oo)-3\pi R(\oo))
\end{equation}
with equality for a ball. Moreover inequality (\ref{ineqcirc})
improves inequality (\ref{circumradius}) when $P(\oo)\ge (3\pi -4)
R(\oo)$.
\end{proposition}

\subsection{Maximum for the perimeter}
Let us consider the following problem
\begin{equation}\label{PB2}
\max_{\substack{\oo\in\C,\\|\oo|\le c}} P(\oo),
\end{equation}
where $c>0$ is a given constant. 
If $\pi a^2\le c\le \pi  b^2$ then a solution exists by the  compactness of the class $\C\cap\{|\oo|\le c\}$ and the continuity of $P(\cdot)$ (for the Hausdorff distance).
In particular using the formulation of the perimeter in terms of the so called \emph{gauge function}, Theorem {2.1} of \cite{LN} guarantees that all the possible solutions are locally polygons in the interior of the annulus $D_b\setminus \overline{D_a}$.

Notice that each solution $\oo_c$ to (\ref{PB2}) in fact saturates
the constraint on the volume, that is $|\oo_c|=c$. Indeed, for every
set $\oo\in\C$ with volume strictly smaller than $c$, there exists
$\oo'\in\C$, with $|\oo'|=c$ and $\oo'\supset \oo$; as $\oo,\oo'$
are planar convex sets, it holds $P(\oo')>P(\oo)$.

Let $\oo_c$ be a solution to (\ref{PB2}) for some fixed $c$; hence $\oo_c$ is a critical point for the functional $\jla{\la}$ with $\la$ corresponding to a Lagrange multiplier associated to the area constraint. 
However $\oo_c$ is not necessarily a minimum for it. 
In particular, as shown in the graph below (see Figure \ref{volume-la}), there are many values of $c\in(\pi a^2,\pi b^2)$ for which there is no solution to (\ref{PB}) of volume $c$, and hence an optimal set to (\ref{PB2}) for those values of $c$ cannot be a solution to (\ref{PB}). 
\begin{center}
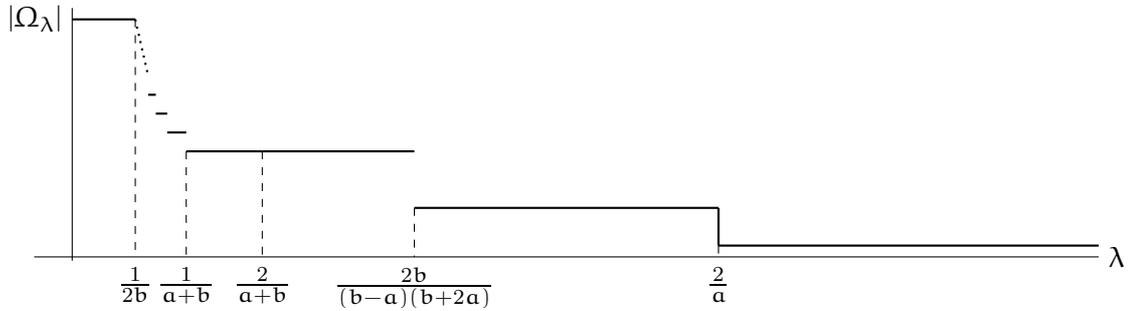
\begin{figure}[h]
\begin{tikzpicture}[x=5cm,y=0.5cm]
\draw (0,1.6)--(0,8.3);\draw (0,8)node[left]{$|\ola|$};
\draw (-0.1,1.7)--(2.7,1.7) node[right]{$\la$};
\draw[thick] (0,8)--(0.166,8);\draw[dashed] (0.166,8)--(0.166,1.7) node[below]{$\frac 1{2b}$};
\draw[thick, dotted] (0.166,8)--(0.2,6.5);
\draw[thick] (0.2, 6)--(0.22,6);\draw[thick] (0.22, 5.5)--(0.25,5.5);\draw[thick] (0.25,5)--(0.3,5);
\draw[thick] (0.3,4.5)--(0.5,4.5);\draw[dashed] (0.3,4.5)--(0.3,1.7) node[below]{$\frac 1{a+b}$};
\draw[thick] (0.5, 4.5)--(0.9,4.5);\draw[dashed] (0.5,4.5)--(0.5,1.7)
node[below]{$\frac 2{a+b}$};\draw[dashed] (0.9,3)--(0.9,1.7) node[below]{$ \frac{2b}{(b-a)(b+2a)}$};
\draw[thick] (0.9, 3)--(1.7,3);\draw[thick] (1.7,3)--(1.7,2);\draw[dashed] (1.7,2)--(1.7,1.7) node[below]{$\frac 2a$};
\draw[thick] (1.7, 2)--(2.7,2);
\end{tikzpicture}
\caption{Graph of the {possible values} of the volume of solutions to (\ref{PB}), as $\la$ varies.}\label{volume-la}
\end{figure}
\end{center}
The main difference between the two problems is that in Problem
(\ref{PB2})  solutions are not necessarily polygons and hence they
could contain parts of arcs of $D_b$ and $D_a$, as explained below.
Notice that, in fact, the proof of Theorem \ref{teopolyg} does not
work for Problem (\ref{PB2}) as the considered perturbations do not
preserve the volume.

As an example, consider the case of a fixed volume closed to that of
the ball $D_b$: $c=\pi b^2-\ep$, for some positive small $\ep$. The
class of sets belonging to $\C$ with volume equal to $c$ only
contains sets closed to the ball $D_b$ and hence each possible side
is not tangent to the interior ball $D_a$. This allows us to assume
that each side of the  boundary is a chord of $D_b$ since otherwise
a technique of parallel chord movements would increase the
perimeter. Hence if a polygon is a critical point for Problem
(\ref{PB2}), the first order conditions (\ref{1ordine}) hold and
they imply that the polygon has at most two different values for its
central angles: $x,y$ with $x>y$. In particular, following Remark
\ref{rmkoptimality}, we can check that the second order optimality
conditions guarantee that there are at most two copies of the angle
$y$ (we have here two equality constraints, thus the critical cone
is of codimension 2). Hence a possible critical polygon for
(\ref{PB2}) is determined by its central angles as $q$ copies of an
angle $x$ with either zero, one or two copies of an angle $y<x$; the
value of the central angles are established using the volume
constraint.

However direct computations show that all the possible critical
polygons  have a perimeter less than the set $\oo_c$ whose boundary
is composed by an arc of the circle $D_b$ and a chord of $D_b$ and
hence for values of $c$ closed to $\pi b^2$, solutions to Problem
(\ref{PB2}) are not polygons.


\section{Appendix}
A list of values for the constants $\hat\beta_N$ of Proposition \ref{regularinscrit}.

\begin{center}
\begin{table}[h]
\begin{tabular}{|r|c|}\hline
N & $\hat\beta_N$\\\hline
3 & 0.32862\\\hline
4 & 0.29260\\\hline
5 & 0.27706\\\hline
6 & 0.26881\\\hline
7 & 0.26388\\\hline
8 & 0.26068\\\hline
9 & 0.25848\\\hline
10 & 0.25690\\\hline
11 & 0.25572\\\hline
12 & 0.25483\\\hline
\end{tabular}
\hspace{0.5cm}
\begin{tabular}{|r|c|}\hline
N & $\hat\beta_N$\\\hline
13 & 0.25413\\\hline
14 & 0.25357\\\hline
15 & 0.25312\\\hline
16 & 0.25275\\\hline
17 & 0.25244\\\hline
18 & 0.25218\\\hline
19 &  0.25196\\\hline
20 & 0.25177\\\hline
21 & 0.25161\\\hline
22 & 0.25147\\\hline
\end{tabular}
\hspace{0.5cm}
\begin{tabular}{|r|c|}\hline
N & $\hat\beta_N$\\\hline
23 & 0.25135\\\hline
24 &  0.25124\\\hline
25 & 0.25114\\\hline
26 & 0.25106\\\hline
27 &  0.25098\\\hline
28 & 0.25091\\\hline
29 & 0.25085\\\hline
30 & 0.25080\\\hline
31 & 0.25075\\\hline
32 & 0.25070\\\hline
\end{tabular}
\hspace{0.5cm}
\begin{tabular}{|r|c|}\hline
N & $\hat\beta_N$\\\hline
33 & 0.25066\\\hline
34 & 0.25062\\\hline
35 & 0.25059\\\hline
36 & 0.25056\\\hline
37 & 0.25053\\\hline
38 & 0.25050\\\hline
39 &  0.25048\\\hline
40 & 0.25045\\\hline
41 & 0.25043\\\hline
42 & 0.25041\\\hline
\end{tabular}
\hspace{0.5cm}
\begin{tabular}{|r|c|}\hline
N & $\hat\beta_N$\\\hline
43 &  0.25039\\\hline
44 & 0.25037\\\hline
45 & 0.25036\\\hline
46 &  0.25034\\\hline
47 & 0.25033\\\hline
48 & 0.25032\\\hline
49 & 0.25030\\\hline
50 & 0.25029\\\hline
51 & 0.25028\\\hline
52 & 0.25027\\\hline
\end{tabular}
\end{table}\label{tableBn}
\end{center}
%

\section*{Acknowledgements}
{This work was done while the first author was supported by the ANR CNRS project GAOS (Geometric Analysis of Optimal Shapes) at the Institut Elie Cartan Nancy.
She would like to thank all the group for their warm welcome.}

\end{document}